\documentclass[12pt]{amsart}
\usepackage{amssymb,amsmath,amsthm,amsfonts,mathrsfs,graphicx}
\usepackage{mathtools}
\usepackage[margin=3cm]{geometry}
\usepackage[colorlinks]{hyperref}
\usepackage{tikz}
\mathtoolsset{showonlyrefs=true}

\title[The Euler-alignment system]{Asymptotic behaviors for the compressible Euler system with nonlinear velocity alignment}

\author[McKenzie Black]{McKenzie Black}
\address[McKenzie Black]{\newline Department of Mathematics, \ 
 University of South Carolina, 1523 Greene St., Columbia, SC 29208, USA}
\email{mmblack@email.sc.edu}

\author[Changhui Tan]{Changhui Tan}
\address[Changhui Tan]{\newline Department of Mathematics, \ 
 University of South Carolina, 1523 Greene St., Columbia, SC 29208, USA}
\email{tan@math.sc.edu}

\thanks{\textit{Acknowledgment.} This work has been supported by the NSF grant DMS 2108264 and a UofSC SPARC grant.}



\newtheorem{theorem}{Theorem}[section]

\newtheorem{corollary}[theorem]{Corollary}
\newtheorem{proposition}[theorem]{Proposition}

\theoremstyle{definition}
\newtheorem{definition}{Definition}[section]

\theoremstyle{remark}
\newtheorem{remark}{Remark}[section]

\numberwithin{equation}{section}

\def\R{\mathbb{R}}
\def\pa{\partial}
\def\A{\mathbf{A}}
\def\u{\mathbf{u}}
\def\w{\mathbf{w}}
\def\x{\mathbf{x}}
\def\y{\mathbf{y}}
\def\z{\mathbf{z}}
\def\grad{\nabla}
\def\div{\grad\cdot}

\def\phim{\underline{\phi}}
\def\D{\mathcal{D}}
\def\Dbar{\overline{\D}}
\def\Dt{D}
\def\Dtt{\widetilde{D}}
\def\Dtbar{\overline{\Dt}}
\def\Dtb{\underline{\Dt}}
\def\V{\mathcal{V}}
\def\Vt{V}
\def\Vtt{\widetilde{V}}
\def\Vtbar{\overline{\Vt}}
\def\Vb{\underline{\V}}
\def\Vtb{\underline{\Vt}}

\begin{document}
\allowdisplaybreaks

\begin{abstract}
We consider the compressible Euler system with a family of nonlinear velocity alignments. The system is a nonlinear extension of the Euler-alignment system in collective dynamics. We show the asymptotic emergent phenomena of the system: alignment and flocking. Different types of nonlinearity and nonlocal communication protocols are investigated, resulting in a variety of different asymptotic behaviors.  
 \end{abstract}

\maketitle 

\tableofcontents

\section{Introduction}\label{sec:intro}
In this paper, the point of concern is the following pressureless Euler system with alignment interactions
\begin{equation} \label{eq:EAS}
	\begin{cases}
	\pa_t\rho+\nabla\cdot(\rho\u)= 0,\\
    \pa_t(\rho\u)+\div(\rho\u\otimes\u)= \rho\A[\rho,\u],
	\end{cases}
\end{equation}
where the density $\rho : \R^d\times\R_+\to\R$ and the momentum $\rho\u : \R^d\times\R_+\to\R^d$. The nonlocal alignment force $\A[\rho,\u]$ takes the form 
\begin{equation}\label{eq:alignment}
	\A[\rho,\u](\x,t)=\int_{\R^d} \phi(\x-\y)\Phi(\u(\y,t)-\u(\x,t))\rho(\y,t)\,d\y.
\end{equation}
The function $\phi$ is known as the \emph{communication protocol}. It measures the strength of the pairwise alignment interaction. We naturally assume that $\phi$ is radially symmetric and decreasing along the radial direction.

The mapping $\Phi : \R^d\to\R^d$ describes the type of alignment. One typical choice is the linear mapping $\Phi(\z)=\z$. The corresponding system \eqref{eq:EAS}-\eqref{eq:alignment} is known as the pressure-less Euler-alignment system.

\subsection{The Euler-alignment system}
The Euler-alignment system arises as the macroscopic description of the celebrated Cucker-Smale model \cite{cucker2007emergent} for animal flocks
\begin{equation}\label{eq:CuckerSmale}
 \begin{cases}
  	\dot{x}_i=v_i,\\
  	\dot{v}_i=\displaystyle\frac{1}{N}\sum_{j=1}^N\phi(x_i-x_j)(v_j-v_i).
 \end{cases}
\end{equation}
Here $\{x_i,v_i\}_{i=1}^N$ denotes the locations and velocities of the $N$ agents. The Euler-alignment system can be derived from \eqref{eq:CuckerSmale} via a kinetic description, see e.g. \cite{ha2008particle,figalli2018rigorous}.

The Euler-alignment system has been extensively studied in the past decade. The global wellposedness theory has been established for different types of communication protocols. When $\phi$ is bounded and Lipschitz, a \emph{critical threshold phenomenon} was discovered in \cite{tadmor2014critical}: subcritical initial data lead to globally regular solutions, while supercritical initial data lead to finite time shock formations. In one dimension, a sharp threshold condition was found in \cite{carrillo2016critical}; while in higher dimensions, sharp results are only available for uni-directional \cite{lear2022existence} and radial \cite{tan2021eulerian} flows.

Another interesting type of communication protocol is when $\phi$ is \emph{singular}, namely
\begin{equation}\label{eq:singularphi}
 \phi(r) = r^{-\alpha},
\end{equation}
with $\alpha>0$. In particular, when $\phi$ is \emph{strongly singular} with $\alpha=d+2s>d$, the alignment operator is closely related to the fractional Laplacian $(-\Delta)^s$, bringing a regularization effect to the solution. 
In one-dimensional periodic domain, global regularity is proved for all non-vacuous initial data in \cite{shvydkoy2017eulerian} for $s\in[\frac12,1)$ and in \cite{do2018global} for $s\in(0,\frac12)$. The result has been extended to general communication protocols that behave like \eqref{eq:singularphi} near the origin, see e.g. \cite{kiselev2018global,miao2021global}. The effect of the vacuum is discussed in \cite{tan2019singularity,arnaiz2021singularity}.
In higher dimensions, global wellposedness result is only known for small initial data \cite{shvydkoy2019global,danchin2019regular}.

When $\phi$ is \emph{weakly singular} with $\alpha\in(0,d)$, the global behavior is known to be similar to the bounded Lipschitz case. A slightly different critical threshold is obtained in \cite{tan2020euler}. It has been further discussed in \cite{leslie2020lagrangian}. The borderline case $\alpha=d$ is studied in \cite{an2020global}.

We shall mention that another active branch on the wellposedness theory for the Euler-alignment system is to incorporate pressure. See e.g. \cite{choi2019global,constantin2020entropy, tong2020global, chen2021global, bai2022global, tadmor2022swarming}. For more results on the Euler-alignment system, we refer to the recent book by Shvydkoy \cite{shvydkoy2021dynamics}.

\subsection{Alignment and flocking}
The Euler-alignment system exhibits remarkable asymptotic behaviors: \emph{alignment} and \emph{flocking}. These collective behaviors are inherited from the Cucker-Smale model \eqref{eq:CuckerSmale}. The mathematical representation of hydrodynamic alignment and flocking are defined as follows.

Let $(\rho,\u)$ be a solution to the system \eqref{eq:EAS}-\eqref{eq:alignment}. Let us define the spatial diameter $\D$ and velocity diameter $\V$ as follows:
\begin{equation}\label{eq:DV}
  \D(t)=\text{diam}\big(\text{supp} \, \rho(\cdot,t)\big)
  \quad \text{and} \quad
  \V(t)=\underset{\x,\y\in\text{supp}(\rho(\cdot,t))}{\sup }|\u(\x,t)-\u(\y,t)|.
\end{equation}
The long time collective behaviors of the system can be identified from the following two concepts:
\begin{itemize}
 \item[(i).] \emph{Flocking}: spatial diameter is bounded in all time, namely there exists a constant $\Dbar<\infty$ such that
\begin{equation}\label{eq:flocking}
\D(t)\leq\Dbar,\quad\forall~t\geq0.	
\end{equation}
 \item[(ii).] \emph{Alignment}: the asymptotic velocity is a constant, or equivalently, velocity diameter decays to zero
 \begin{equation}\label{eq:alignmentdecay}
  \lim_{t\to\infty}\V(t)=0.
 \end{equation}
\end{itemize}

We say the flocking and alignment are \emph{unconditional} if \eqref{eq:flocking} and \eqref{eq:alignmentdecay} hold for all initial data;
we say the flocking and alignment are \emph{conditional} if whether \eqref{eq:flocking} and \eqref{eq:alignmentdecay} hold depend on initial data: subcritical initial data lead to flocking and alignment, while supercritical initial data lead to no flocking and no alignment.  In addition, we introduce the following new concept.
\begin{definition}[Semi-unconditional flocking and alignment]\label{def:semi}
 We say flocking and alignment are \emph{semi-unconditional} if \eqref{eq:flocking} and \eqref{eq:alignmentdecay} hold for subcritical initial density, and for any initial velocity. More precisely, there exists a subcritical region on $\D_0$, for any $\V_0$, such that \emph{semi-unconditional} if \eqref{eq:flocking} hold.
\end{definition}

The flocking property for the Cucker-Smale model \eqref{eq:CuckerSmale} has been studied in \cite{ha2009simple,motsch2011new}. The same phenomenon is shown for strong solutions to the Euler-alignment system in \cite{tadmor2014critical} (see \cite{leslie2021sticky} for results on weak solutions). Interestingly, the asymptotic behaviors vary for different communication protocols \eqref{eq:singularphi}, particularly on the behaviors of $\phi$ near infinity. 
If $\alpha\in[0,1]$, the communication protocol has a \emph{fat tail}, the solution has unconditional flocking and alignment properties. Moreover, $\Vt(t)$ decays exponentially in time (known as \emph{fast alignment}). If $\alpha>1$, the communication protocol has a \emph{thin tail}, the flocking and alignment are conditional. See Scenario 0 (S0) in Figure \ref{fig:main} for more details. 

The flocking behavior for the Euler-alignment system has been further investigated in \cite{shvydkoy2017eulerian2}. They showed \emph{fast flocking}: density $\rho(\x,t)$ converges to a traveling wave solution $\rho_\infty(\x-t\bar{\u})$ exponentially in time. See \cite{leslie2019structure,lear2022geometric,leslie2021sticky,bai2022global} for more development.
 
\subsection{Nonlinear velocity alignment}
We consider a new family of alignment interactions \eqref{eq:alignment}, where the mapping $\Phi$ takes the form
\begin{equation}\label{eq:Phi}
	\Phi(\z)=|\z|^{p-2}\z.
\end{equation}
In particular when $p=2$, $\Phi$ is linear and the system \eqref{eq:EAS}-\eqref{eq:alignment} reduces to the Euler-alignment system.

The nonlinear velocity alignment was introduced in \cite{ha2010emergent} for the agent-based Cucker-Smale type dynamics. The system \eqref{eq:EAS}-\eqref{eq:alignment} was derived and studied recently in \cite{tadmor2022swarming, lu2022hydrodynamic} as a formal hydrodynamic representation of the agent-based model, named \emph{$p$-alignment hydrodynamics}.

One motivation of considering the nonlinearity in \eqref{eq:Phi} is its natural connection to the fractional $p$-Laplacian 
\begin{equation}
 (-\Delta)_p^s\u(\x)= c_{s,p,d}\, P.V.\int_{\R^d}\frac{\Phi(\u(\x)-\u(\y))}{|\x-\y|^{d+2s}}\,d\y.
\end{equation}
Indeed, if we take a strongly singular communication protocol \eqref{eq:singularphi} with $\alpha=d+2s$ and enforce the density $\rho\equiv1$, then the nonlinear velocity alignment acts like fractional $p$-Laplacian
\[c_{s,p,d}\,\A[1,\u](\x,t)= -(-\Delta)_p^s\u(\x,t).\]
Its nonlocal and nonlinear feature has drawn a lot of attentions lately. The fractional $p$-Laplacian evolution equation
\[\pa_t\u=(-\Delta)_p^s\u\]
has been extensively studied in a recent series of works by V\'asquez \cite{vazquez2016dirichlet,vazquez2020evolution,vazquez2021fractional,vazquez2022growing}.

\subsection{Main results}
In this paper, we study the Euler system \eqref{eq:EAS}-\eqref{eq:alignment} with nonlinear velocity alignment \eqref{eq:Phi}. While the global wellposedness is an interesting problem of its own, our focus here is on the asymptotic behavior of the system.

The focus of this paper is on the asymptotic behavior of the Euler system \eqref{eq:EAS}-\eqref{eq:alignment} with nonlinear velocity alignment \eqref{eq:Phi}. As discovered in \cite{tadmor2022swarming}, the nonlinearity leads to a fruitful of diverse alignment and flocking behaviors. In particular, the convergence of the velocity diameter in \eqref{eq:alignmentdecay} has a polynomial decay in time, in oppose to the linear alignment $p=2$, where the decay rate is exponential.

\renewcommand{\arraystretch}{1.2}
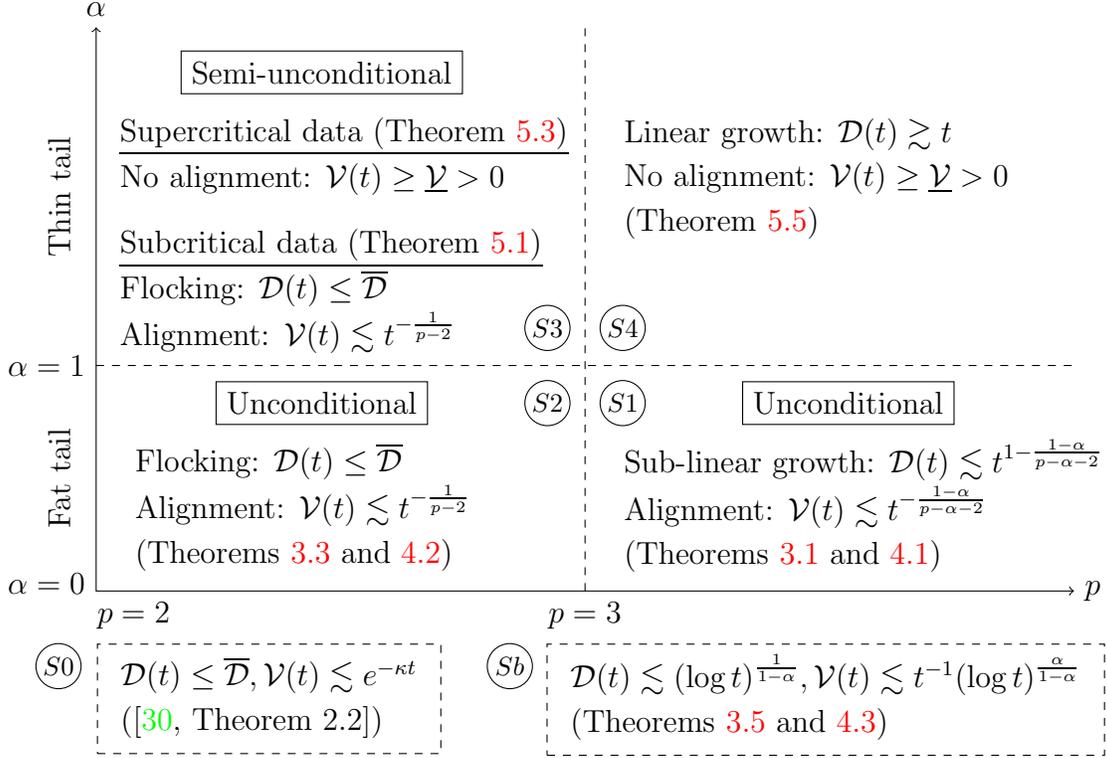
\begin{figure}[ht]
\begin{tikzpicture}
 \draw[<->] (0,7.5) node[above]{$\alpha$} -- (0,0) node[left,yshift=1mm]{$\alpha=0$} -- (13,0) node[right]{$p$};
 \draw[dashed] (0,3) node[left]{$\alpha=1$} -- (13,3);
 \node[below] at (0.5,0) {$p=2$};
 \draw[dashed] (6.5,0) node[below]{$p=3$} -- (6.5,7.5);
 \node[rotate=90, above] at (-.2,1.5) {Fat tail};
 \node[rotate=90, above] at (-.2,5.25) {Thin tail};
 \node[draw, dashed, anchor=north] at (2.3,-.7) {
 $\begin{array}{l}
   \D(t)\leq\Dbar, \V(t)\lesssim e^{-\kappa t}\\
   \text{(\cite[Theorem 2.2]{tadmor2014critical})}
 \end{array}$
 };
 \draw (-.5,-1) circle [radius=0.3] node {\footnotesize$S0$};
 \node[draw] at (3,2.5) {Unconditional};
 \node[anchor=west] at (.2,1.1) {
 $\begin{array}{l}
	\text{Flocking: }\D(t)\leq\Dbar\\
	\text{Alignment: } \V(t)\lesssim t^{-\frac{1}{p-2}}\\
	\text{(Theorems \ref{thm:case2} and \ref{thm:case2o})}
\end{array}$
 };
 \draw (6,2.5) circle [radius=0.3] node {\footnotesize$S2$};
 \node[draw, dashed, anchor=north] at (9.7,-.7) {
 $\begin{array}{l}
   \D(t)\lesssim (\log t)^{\frac{1}{1-\alpha}}, \V(t)\lesssim t^{-1}(\log t)^{\frac{\alpha}{1-\alpha}}\\
   \text{(Theorems \ref{thm:case3} and \ref{thm:case3o})}
 \end{array}$
 };
 \draw (5.5,-1) circle [radius=0.3] node {\footnotesize$Sb$};
 \node[draw] at (10,2.5) {Unconditional};
 \node[anchor=west] at (6.7,1.1) {
 $\begin{array}{l}
	\text{Sub-linear growth: }\D(t)\lesssim t^{1-\frac{1-\alpha}{p-\alpha-2}}\\
	\text{Alignment: } \V(t)\lesssim t^{-\frac{1-\alpha}{p-\alpha-2}}\\
	\text{(Theorems \ref{thm:case1} and \ref{thm:case1o})}
\end{array}$
 }; 
 \draw (7,2.5) circle [radius=0.3] node {\footnotesize$S1$};
 \node[draw] at (3,6.9) {Semi-unconditional};
 \node[anchor=west] at (0,4) {
 $\begin{array}{l}
    \text{\underline{Subcritical data (Theorem \ref{thm:case4})}}\\ 
	\text{Flocking: }\D(t)\leq\Dbar\\
	\text{Alignment: } \V(t)\lesssim t^{-\frac{1}{p-2}}
\end{array}$
 };  
 \node[anchor=west] at (0,5.8) {
 $\begin{array}{l}
	\text{\underline{Supercritical data (Theorem \ref{thm:case4o})}}\\
	\text{No alignment: } \V(t)\geq\Vb>0
\end{array}$
 }; 
 \draw (6,3.5) circle [radius=0.3] node {\footnotesize$S3$};
 \node[anchor=west] at (6.7,5.5) {
 $\begin{array}{l}
 	\text{Linear growth: } \D(t)\gtrsim t\\
	\text{No alignment: } \V(t)\geq\Vb>0\\
	\text{(Theorem \ref{thm:case5o})}
\end{array}$
 }; 
\draw (7,3.5) circle [radius=0.3] node {\footnotesize$S4$};
\end{tikzpicture}
\caption{Gallery of results.} \label{fig:main}
\end{figure}

Figure \ref{fig:main} is a collection of asymptotic behaviors of the system \eqref{eq:EAS}-\eqref{eq:alignment} with different nonlinear velocity alignment parametrized by $p$, and different types of communication protocols parameterized by $\alpha$. Our results are summarized as follows.

\textit{Scenario 1 (S1): $p>3, 0\leq\alpha<1$.}
The system has unconditional alignment \eqref{eq:alignmentdecay} with polynomial decay rate $\frac{1-\alpha}{p-\alpha-2}$. Due to the strong nonlinearity, there is no guaranteed flocking. But the spatial diameter has a sub-linear growth. See Theorem \ref{thm:case1} for detailed descriptions. We further show in Theorem \ref{thm:case1o} that the decay rate on $\V(t)$ and the growth rate on $\D(t)$ are optimal.

\textit{Scenario 2 (S2): $2<p<3, 0\leq\alpha<1$.}
The system has unconditional flocking \eqref{eq:flocking}, and alignment \eqref{eq:alignmentdecay} with polynomial decay rate $\frac{1}{p-2}$ (Theorem \ref{thm:case2}). Moreover, the decay rate on $\V(t)$ is optimal (Theorem \ref{thm:case2o}).

\textit{Borderline Scenario (Sb): $p=3, 0\leq\alpha<1$.} 
The spatial diameter can have a logarithmic growth. This is also a logarithmic correction to the decay on the velocity diameter (Theorem \ref{thm:case3}). The rates are optimal (Theorem \ref{thm:case3o}).

\textit{Scenario 3 (S3): $2<p<3, \alpha>1$.}
The asymptotic behaviors are \emph{conditional}. For subcritical initial data, the system exhibits flocking and alignment with the same rate as in Scenario 2 (Theorem \ref{thm:case4}). On the other hand, there are supercritical initial data that lead to \emph{no alignment}, namely \eqref{eq:alignmentdecay} is violated (Theorem \ref{thm:case4o}). Moreover, we show that the flocking and alignment are \emph{semi-unconditional}.

\textit{Scenario 4 (S4): $p>3, \alpha>1$.}
For any initial spatial and velocity diameters $(\D_0,\V_0)$, regardless of how small they are, we construct initial data that lead to no alignment (Theorem \ref{thm:case5o}). 
\medskip

The polynomial decay in time for the system was discovered by Tadmor in a very recent work \cite{tadmor2022swarming}, covering Scenarios 1 and 2. Our results in Theorems \ref{thm:case1} and \ref{thm:case2} echo with the findings. Moreover, we show that the decay rates are optimal, as well as an explicit logarithmic correction in the borderline case $p=3$.

The flocking behavior of the agent-based Cucker-Smale type dynamics ($2<p<3$) was investigated in \cite{ha2010emergent}, using a smartly chosen Lyapunov functional that was first introduced in \cite{ha2009simple}. This can be applied to Scenarios 2 and 3 in our system. See Section \ref{sec:Lyap} for details of this approach. The asymptotic behaviors for general choices of $p$ was studied in \cite{kim2020complete}. The result seems to depend on the number of agents $N$, and can not be extended to the macroscopic system (with $N\to\infty$).

Our approach makes use of the \emph{method of invariant region}. The idea is to construct an invariant region to the rescaled spatial and velocity diameters, and show that the relevant quantities stay inside the region in all time. Compared with the Lyapunov functional approach, our method can cover the cases when the nonlinearity is strong ($p\geq3$). It can also be used to detect the \emph{no alignment} property.

We would like to highlight our results in Scenario 3. With a thin tail, the system exhibits conditional flocking. This has been proved in \cite{ha2010emergent} using the Lyapunov functional approach. See Theorem \ref{thm:Lyap} for the full description. We show a surprising result that the flocking is \emph{semi-unconditional}: the subcritical region that ensures flocking is independent of the initial velocity. 

Finally, we comment that our results are based on the analysis to a paired inequalities \eqref{eq:ODI}. The framework established by Tadmor in \cite{tadmor2022swarming} works beautifully for general pressure laws. A paired inequalities similar to \eqref{eq:ODI} was derived, using the energy fluctuation $\delta\mathcal{E}$ to replace the velocity diameter $\V$. Thanks to the inequalities on $(\D,\delta\mathcal{E})$, our results can be extended to the compressible Euler system with nonlinear velocity alignment and general choices of pressure.

\subsection{Outline of the paper}
We start with presenting a collection of preliminaries in Section \ref{sec:prelim}, including the derivation of the paired inequalities \eqref{eq:ODI} and some related results in the literature. In Section \ref{sec:fat}, we study the asymptotic behaviors of our system when the communication protocol has a fat tail. This covers the results in Scenarios 1 and 2, as well as the borderline scenario. We then show in Section \ref{sec:sharp} that the quantitative rates of decay or growth that we obtained are sharp. Finally, Section \ref{sec:thin} is devoted to Scenarios 3 and 4, when the communication protocol has a thin tail. In particular, we show semi-unconditional flocking and alignment in Scenario 3.

\section{Preliminaries}\label{sec:prelim}
Let us rewrite our main system equivalently as the evolution of $(\rho,\u)$.
\begin{equation} \label{eq:main}
	\begin{cases}
	\pa_t\rho+\nabla\cdot(\rho\u)= 0,\\
    \pa_t\u+\u\cdot\grad\u = \A[\rho,\u],\\
    \A[\rho,\u](\x,t)=\displaystyle\int_{\R^d} \phi(\x-\y)\Phi(\u(\y,t)-\u(\x,t))\rho(\y,t)\,d\y,\quad
    \Phi(\z)=|\z|^{p-2}\z.
	\end{cases}
\end{equation}

\subsection{The paired inequalities}
We start with the derivation of the following paired ordinary differential inequalities on $(\D,\V)$ that play an important role in the analysis of the asymptotic behavior of our system:
\begin{equation} \label{eq:ODI}
  \begin{cases}
   \D'(t) \leq \V(t),\\
   \V'(t) \leq -C\phi(\D(t))\V(t)^{p-1},
  \end{cases}
	\quad\text{with}\quad
  \begin{cases}
   \D(0)=\D_0,\\
   \V(0)=\V_0.
  \end{cases}
\end{equation}
This type of inequalities was first introduced in \cite{ha2009simple} (with $p=2$), in the context of the agent-based Cucker-Smale model, and in \cite{ha2010emergent} for general $p>1$. Using a similar idea, it was derived for the Euler-alignment system ($p=2$) in \cite{tadmor2014critical}. More recently, Tadmor in \cite{tadmor2022swarming} derived \eqref{eq:ODI} from \eqref{eq:main}, not only for any $p>1$, but also adapted general pressure laws.

For the sake of self-consistency, we present a derivation of \eqref{eq:ODI} for our system \eqref{eq:main}, with general choice of $p\in(1,\infty)$.

\begin{proposition}
 Let $p>1$. Suppose $(\rho,\u)$ is a solution to the system \eqref{eq:main}. Define $(\D,\V)$ as in \eqref{eq:DV}. Then, $(\D(t),\V(t))$ are continuous in time, and the paired inequalities \eqref{eq:ODI} hold almost everywhere in $t$.
\end{proposition}

\begin{proof}
 Let us fix a time $t$. Let $\z, \w\in\text{supp}(\rho(\cdot,t))$ such that the maximum velocity diameter is attained, namely
\[\V(t)=|\u(\z,t)-\u(\w,t)|.\]
 Clearly, $\grad\u(\w,t)=\grad\u(\z,t)=0$. Applying \eqref{eq:main}$_2$ and Rademacher's Lemma (e.g. \cite[Lemma 3.5]{shvydkoy2021dynamics}), we obtain
\begin{align*}
 \frac{d}{dt} \V(t)^2=&\,2\big(\u(\z,t)-\u(\w,t)\big)\cdot\big(\pa_t\u(\z,t)-\pa_t\u(\w,t)\big)\\
 =&\,2\big(\u(\z,t)-\u(\w,t)\big)\cdot\big(\A[\rho,\u](\z,t)-\A[\rho,\u](\w,t)\big).
\end{align*}
Next, we work on the alignment force. For simplicity, we shall suppress the $t$-dependence throughout the rest of the proof.
\begin{align*}
 	&\A[\rho,\u](\z)-\A[\rho,\u](\w)\\
	=&\int_{\R^d} \phi(\z-\y)\Phi(\u(\y)-\u(\z))\rho(\y)\,d\y-\int_{\R^d} \phi(\w-\y)\Phi(\u(\y)-\u(\w))\rho(\y)\,d\y\\
	=&\int_{\R^d} \big(\phi(\z-\y)-\eta\big)\Phi(\u(\y)-\u(\z))\rho(\y)\,d\y\\
	-&\int_{\R^d} \big(\phi(\w-\y)-\eta\big)\Phi(\u(\y)-\u(\w))\rho(\y)\,d\y\\
	&+\eta\int_{\R^d}\big(\Phi(\u(\y)-\u(\z))-\Phi(\u(\y)-\u(\w))\big)\rho(\y)\,d\y.
\end{align*}
Here, we take $\eta=\phi(\D(t))$ so that $\phi(\z-\y)-\eta>0$ and $\phi(\w-\y)-\eta>0$.

Since $\Phi$ is odd and increasing, and $\w, \z$ are where the maximum is attained, we have
\[\big(\u(\z)-\u(\w)\big)\cdot\Phi(\u(\y)-\u(\z))\leq0,\quad
\big(\u(\z)-\u(\w)\big)\cdot\Phi(\u(\y)-\u(\w))\geq0,\]
for any $\y\in\text{supp}(\rho)$.
Therefore, we have
\[\frac{d}{dt}\V(t)^2\leq 2\big(\u(\z)-\u(\w)\big)\cdot\eta\int_{\R^d}\big(\Phi(\u(\y)-\u(\z))-\Phi(\u(\y)-\u(\w))\big)\rho(\y)\,d\y.\]
Take our $\Phi(\z)=|\z|^{p-2}\z$ in \eqref{eq:Phi}. When $p>1$, elementary calculus implies the following bound
\[
\big(\u(\z)-\u(\w)\big)\cdot\big(\Phi(\u(\y)-\u(\z))-\Phi(\u(\y)-\u(\w))\big)\leq -2^{2-p}|\u(\z)-\u(\w)|^p, 
\]
for any $\y\in\text{supp}(\rho)$, where the equality is achieved when $\y=\frac{\z+\w}{2}$. 
Apply the bound and we get
\[\frac{d}{dt}\V(t)^2\leq -2^{3-p}\eta\,\V(t)^p\int_{\R^d}\rho(\y)\,d\y.\]
Note that the total mass $\int_{\R^d}\rho(\y)\,d\y$ is conserved in time. We conclude with
\[\V'(t)\leq -C\phi(\D(t))\V(t)^{p-1},\quad\text{where}\quad C=-2^{2-p}\int_{\R^d}\rho_0(\y)\,d\y.\]
\end{proof}

\subsection{Global communication}
 One scenario where the global behavior can be easily obtained is when the communication protocol has a positive lower bound, 
\begin{equation}\label{eq:alpha0}
 \phi(r)\geq\phim>0,\quad\forall~r\geq0. 	
\end{equation}
In this case, \eqref{eq:ODI} implies the following results.
\begin{theorem}\label{thm:alpha0}
Let $p>1$ and $\phi$ satisfies \eqref{eq:alpha0}. 
Take any bounded $(\D_0,\V_0)$. Suppose $(\D,\V)$ satisfies \eqref{eq:ODI}. Then, we have
\begin{itemize}
 \item If $1<p<2$, there exists a finite time $T_*$ such that $\lim_{t\to T_*}\V(t)=0$.
 \item If $p=2$, then $\V(t)$ decays to zero exponentially in time, $\V(t)\lesssim e^{-\kappa t}$.
 \item If $p>2$, then $\V(t)$ decays to zero algebraically in time, $\V(t)\lesssim t^{-\beta_*}$, with the decay rate $\beta_*=\frac{1}{p-2}$.
\end{itemize}
Moreover, we have
\begin{itemize}
 \item If $1<p<3$, the solution flocks.
 \item If $p=3$, $\D(t)$ has logarithmic growth in time, $\D(t)\lesssim\log t$.	
 \item If $p>3$, $\D(t)$ has sublinear growth in time, $\D(t)\lesssim t^{1-\beta_*}$.
\end{itemize}
\end{theorem}
\begin{proof}
Since $\phi$ is lower bounded, we apply \eqref{eq:ODI}$_2$ and get
\[
 \V'(t)\leq-C\phim\,\V^{p-1}(t).
\]
For $p=2$, we have the exponential decay 
\[\V(t)\leq \V_0 e^{-C\phim\,t}.\]
For $p\neq2$, separation of variable yields
\[\V(t)\leq \big(\V_0^{2-p}-(2-p)C\phim\, t\big)^{\frac{1}{2-p}}.\]
When $p<2$, we have $\V(T_*)=0$ at $T_*=\frac{\V_0^{2-p}}{(2-p)C\phim}$. When $p>2$, we get $\V(t)\lesssim t^{-\frac{1}{p-2}}$.

For $\D(t)$, we plug in the bounds on $\V$ to the integral form of \eqref{eq:ODI}$_1$
\[\D(t)\leq \D(0)+\int_0^t\V(\tau)\,d\tau.\]
When $p<3$, $\int_0^\infty\V(\tau)\,d\tau$ converges and hence $\D(t)$ is bounded uniformly in time.
When $p\geq 3$, we have
\[\int_0^t\V(s)\,ds\lesssim\begin{cases}
 t^{\frac{p-3}{p-2}}&p>3,\\ \log t & p=3.	
\end{cases}
\] 
Therefore, $\D(t)$ has a bound that grows sub-linearly in time. 
\end{proof}

Note that the results hold for any initial data. Therefore, the alignment and flocking properties are \emph{unconditional}.

\subsection{Flocking via Lyapunov functional}\label{sec:Lyap}
A more interesting scenario is when the communication protocol $\phi(r)$ decays to zero as $r\to\infty$. In particular, we consider $\phi(r)\sim r^{-\alpha}$ near infinity with $\alpha>0$, namely there exist positive constants $\lambda<\Lambda$ and $R$ such that
\begin{equation}\label{eq:alpha}
 \lambda r^{-\alpha}\leq \phi(r)\leq \Lambda r^{-\alpha},\quad\forall~r\geq R. 	
\end{equation}
This scenario has been studied in \cite{ha2009simple} when the alignment operator $\A[\rho,\u]$ is linear in $\u$, namely the $p=2$ case. The result has been extended to $2<p<3$ in \cite{ha2010emergent}. The flocking behavior \eqref{eq:flocking} is obtained, by brilliantly introducing a Lyapunov functional 
\begin{equation}\label{eq:Lyapnov}
	\mathcal{E}(t)=\V^{3-p}(t)+(3-p)C\psi(\D(t)),\quad \psi(\D(t)):=\int_{\D_0}^{\D(t)}\phi(r)\,dr.
\end{equation}
One can check that 
\[\mathcal{E}'(t)\leq (3-p)\V^{2-p}\cdot\big(-C\phi(\D(t))\V(t)^{p-1}\big)+(3-p)C\phi(\D(t))\cdot\V(t)=0.\]
This leads to $\mathcal{E}(t)\leq\mathcal{E}(0)$, and in particular
\[(3-p)C\psi(\D(t))\leq \V_0^{3-p}\quad\Rightarrow\quad \D(t)\leq \psi^{-1}\left(\frac{\V_0^{3-p}}{(3-p)C}\right).\] 

If the communication protocol $\phi$ has a fat tail, i.e. $\phi$ is non-integrable at infinity, the range of $\psi$ covers $\R_+$. Hence, $\psi^{-1}$ is well-defined for any $\V_0$. This leads to unconditional flocking.

If the communication protocol $\phi$ has a thin tail, i.e. $\phi$ is integrable at infinity, the range of $\psi$ contains $[0, \int_{\D_0}^\infty\phi(r)\,dr)$. Then flocking is guaranteed if
\[\frac{\V_0^{3-p}}{(3-p)C}<\int_{\D_0}^\infty\phi(r)\,dr.\]

We summarize the results as follows.
\begin{theorem}[{\cite{ha2009simple, ha2010emergent}}]\label{thm:Lyap}
 Let $2\leq p<3$ and $\phi$ satisfies \eqref{eq:alpha}. Suppose $(\D,\V)$ satisfies \eqref{eq:ODI}. Then,
 \begin{itemize}
  \item For fat tail communication $\alpha<1$: for any initial data $(\D_0,\V_0)$, the flocking property \eqref{eq:flocking} holds, with
  	\[\Dbar=\left(\D_0^{1-\alpha}+\frac{1-\alpha}{(3-p)\lambda C}\V_0^{3-p}\right)^{\frac{1}{1-\alpha}}.\]
  \item For thin tail communication $\alpha>1$: the flocking property \eqref{eq:flocking} holds when the initial data $(\D_0,\V_0)$ satisfies
  \begin{equation}\label{eq:Lyapsub}
  	\D_0\V_0^{\frac{3-p}{\alpha-1}}\leq\left(\frac{(3-p)\lambda C}{\alpha-1}\right)^{\frac{1}{\alpha-1}}.
  \end{equation} 
 \end{itemize}
\end{theorem}
Once the flocking property is shown, one can apply Theorem \ref{thm:alpha0} with $\phim=\phi(\Dbar)$ and obtain alignment with polynomial decay rate $\frac{1}{p-2}$.

The Lyapunov functional approach is simple and elegant. However, in the case when $\alpha>1$, the sufficient condition \eqref{eq:Lyapsub} depends on $\V_0$. Therefore, the resulting flocking behavior is not \emph{semi-unconditional}. We will show an improved result of semi-unconditional flocking in Theorem \ref{thm:case4}.

 \section{Fat tail communications: Unconditional flocking and alignment}\label{sec:fat}
 In this section, we study the asymptotic behaviors of our system \eqref{eq:main} with fat tail communication protocols $\phi$ that satisfy \eqref{eq:alpha} with $\alpha\in(0,1)$. Our main goal is to analyze the long time behaviors of the paired inequalities \eqref{eq:ODI}.
 
\subsection{Heuristics}
Let us start with a heuristic argument on the asymptotic behaviors of the system. For a simple illustration, we assume the equalities hold in \eqref{eq:ODI}.
 
Suppose $\V(t)\sim t^{-\beta}$ for some $\beta\in(0,1)$. Then, $\D(t)\sim t^{1-\beta}$. The growth of $\D(t)$ will have an effect on the lower bound of $\phi(\D(t))$. Indeed, we have $\phi(\D(t))\V^{p-1}\sim t^{-\alpha(1-\beta)-\beta(p-2)}$. To match the rate of $\V'(t)\sim t^{-\beta-1}$, we should have
 \[-\alpha(1-\beta)-\beta(p-1)=-\beta-1,\quad\text{or equivalently}\quad \beta=\frac{1-\alpha}{p-2-\alpha}.\]
Hence, we expect the following asymptotic behavior
 \[\V(t)\sim t^{-\frac{1-\alpha}{p-2-\alpha}},\quad \D(t)\sim t^{1-\frac{1-\alpha}{p-2-\alpha}}.\]
Note that the rates above are subject to the assumption $\beta<1$, or equivalently $p>3$. 

For $\beta>1$, $\V(t)$ is integrable and therefore $\D(t)\leq\overline\D$ is bounded. Then $\phi(\D(t))$ has a positive lower bound $\phim=\phi(\Dbar)>0$. Theorem \ref{thm:alpha0} suggests that the asymptotic behavior would be
 \[\V(t)\sim t^{-\frac{1}{p-2}},\quad \D(t)\leq\Dbar.\]

The heuristic arguments agree with the asymptotic alignment and flocking behaviors with rates in Figure \ref{fig:main}. The rest of the section is devoted to a rigorous study of the arguments. 
We introduce a method based on constructing \emph{invariant regions} to obtain the desired bounds. Moreover, we will show unconditional alignment and flocking properties to the solutions.
 
\subsection{Scenario 1: Unconditional alignment and sub-linear growth} 
We first state our result on the asymptotic behaviors of $(\D,\V)$ when $p>3$.

\begin{theorem}\label{thm:case1}
Let $p>3$ and $\phi$ satisfies \eqref{eq:alpha} with $\alpha\in[0,1)$. 
Take any bounded $(\D_0,\V_0)$. Suppose $(\D,\V)$ satisfies \eqref{eq:ODI}. Then, we have
\begin{equation}\label{eq:case1bound}
 \D(t)\lesssim t^{1-\frac{1-\alpha}{p-2-\alpha}},\quad \V(t)\lesssim t^{-\frac{1-\alpha}{p-2-\alpha}}.
\end{equation}
\end{theorem}

To prove the theorem, we first scale $(\D,\V)$ according to the expected time scales. Define
\begin{equation}\label{eq:DVscale}
    \Dt(t)=(t+1)^{\beta^*-1} \D(t),\quad\Vt(t)=(t+1)^{\beta^*} \V(t),
\end{equation}
where for simplicity we denote
\begin{equation}\label{eq:betasup}
\beta^*=\frac{1-\alpha}{p-2-\alpha}.	
\end{equation}
Then, the bounds in \eqref{eq:case1bound} hold if $(\Dt,\Vt)$ are bounded.

To control $(\Dt,\Vt)$, we calculate their dynamics using \eqref{eq:ODI}. It yields
\[
\Dt'(t)=(\beta^*-1)(t+1)^{\beta^*-2}\D(t)+(t+1)^{\beta^*-1}\D'(t)
    \leq \frac{1}{t+1} \big((\beta^*-1) \Dt(t) +\Vt(t)\big)
\]
and
\begin{align}
 \Vt'(t)=&\,\beta^*(t+1)^{\beta^*-1}\V(t)+(t+1)^{\beta^*}\V'(t)\label{eq:Vtprime}\\
 \leq&\, \frac{\beta^*}{t+1} \Vt(t)-(t+1)^{\beta^*} C \phi\big((t+1)^{1-\beta^*}\Dt(t)\big)\cdot(t+1)^{-\beta^*(p-1)}\Vt(t)^{p-1}\nonumber\\
 \leq&\,\frac{1}{t+1} \big(\beta^* \Vt(t)-\lambda C\Dt(t)^{-\alpha}\Vt(t)^{p-1}\big).\nonumber
\end{align}
Here, we have used the definition of $\beta^*$ \eqref{eq:betasup} and the assumption on $\phi$ \eqref{eq:alpha} in the last inequality.

To obtain an autonomous system of inequalities, we shall introduce a new time variable
\[\tau=\log(t+1),\]
so that $\frac{d\tau}{dt}=\frac{1}{t+1}$. For simplicity, we still use $(\Dt,\Vt)$ to denote the corresponding functions of $\tau$. This yields the paired inequalities 
 \begin{equation}\label{eq:DtVt1}
   \begin{cases}
    \Dt'(\tau) \leq(\beta^*-1) \Dt(\tau) +\Vt(\tau),\\
    \Vt'(\tau) \leq \beta^* \Vt(\tau)-\lambda C\Dt(\tau)^{-\alpha}\Vt(\tau)^{p-1},
   \end{cases}
	\quad\text{with}\quad
  \begin{cases}
   \Dt(0)=\D_0,\\
   \Vt(0)=\V_0.
  \end{cases}
 \end{equation}
We are left to show that $(\Dt,\Vt)$ are bounded, using the inequalities in \eqref{eq:DtVt1}. Theorem \ref{thm:case1} is proved given the following proposition.

\begin{proposition}\label{prop:invreg}
 Let $(\D_0,\V_0)\in\R_+\times\R_+$. Suppose $(\Dt,\Vt)$ satisfies \eqref{eq:DtVt1}. Then $(\Dt,\Vt)$ are bounded in all time, namely there exists finite constants $\Dtbar$ and $\Vtbar$, depending on $\D_0, \V_0, p, \alpha$, such that
 \begin{equation}\label{eq:DtVtbound1}
  \Dt(\tau)\leq\Dtbar,\quad \Vt(\tau)\leq\Vtbar,\quad\forall~\tau\geq0. 
 \end{equation}
\end{proposition}

We shall remark that Proposition \ref{prop:invreg} works for any initial data $(\D_0,\V_0)$. Therefore, the resulting alignment behavior is \emph{unconditional}.

\begin{proof}[Proof of Proposition \ref{prop:invreg}]
We make use of the \emph{method of invariant region}. The plan is to construct a bounded region in $\R_+\times\R_+$ that contains $(\D_0,\V_0)$, and show that the trajectory of $(\Dt(\tau),\Vt(\tau))$ never exits the region. 

Define
\begin{equation}\label{eq:M}
 M=\max\left\{\V_0,(1-\beta^*)\D_0,\left(\frac{\beta^*}{\lambda C(1-\beta^*)^\alpha}\right)^{\frac{1}{p-2-\alpha}}\right\},
\end{equation}
and consider the region
\begin{equation}\label{eq:invreg}
 A=\left[0,\frac{M}{1-\beta^*}\right]\times\left[0,M\right].
\end{equation}
Figure \ref{fig:DTVT} illustrates the the invariant region.
From the definition, it is easy to see that $(\Dt(0),\Vt(0))\in A$. 

\begin{figure}[ht]
\begin{tikzpicture}[domain=0:5]
 \draw[<->] (0,2.5) node[above]{$\Vt$} -- (0,0) node[below left]{$0$} -- (6,0) node[right]{$\Dt$};
 \draw (0,2) node[left]{$\Vtbar=M$} -- (4,2) -- (4,0) node[below,xshift=6mm]{$\Dtbar=\frac{M}{1-\beta^*}$};
 \draw[blue,->] (1,2) -- (1,1.7);
 \draw[blue,->] (2,2) -- (2,1.7);
 \draw[blue,->] (3,2) -- (3,1.7);
 \draw[red,->] (4,1) -- (3.7,1);
 \draw[red,->] (4,.5) -- (3.7,.5);
 \draw[red,->] (4,1.5) -- (3.7,1.5);
 \draw[red!70,dashed] plot(\x,\x/2) node[right,yshift=2mm] {{\footnotesize$(\beta^*-1) \Dt+\Vt=0$}};
 \draw[blue!70,dashed] plot(\x,{(\x)^(1/2)}) node[right,yshift=-2mm] {{\footnotesize$\beta^* \Vt-\lambda C\Dt^{-\alpha}\Vt^{p-1}=0$}};
 \draw node at (2,1) {$A$};
\end{tikzpicture}
\caption{An illustration of the invariant region $A$ defined in \eqref{eq:invreg}. To the right of the red line, $\Dt'\leq0$. Above the blue curve, $\Vt'\leq0$. Hence, trajectories can not exit $A$.} \label{fig:DTVT}
\end{figure}

We now show that $(\Dt(\tau),\Vt(\tau))\in A$ for all $\tau\geq0$.
Let us argue by contradiction. 
Suppose there exists a finite time $\tau$ such that $(\Dt(\tau),\Vt(\tau))\not\in A$. Then by continuity, there must exists a time $\tau_*$ such that $(\Dt,\Vt)$ exits the region $A$ at $\tau=\tau_*$, namely
\[(\Dt(\tau_*),\Vt(\tau_*))\in \pa A\quad\text{and}\quad
\Dt(\tau_*+),\Vt(\tau_*+))\not\in A.\]
There are two cases.

\textit{Case 1:} $(\Dt,\Vt)$ exits to the right, namely $\Dt(\tau_*)=\frac{M}{1-\beta^*}$, $\Vt(\tau_*)\in[0,M]$, and $\Dt(\tau_*+)>\frac{M}{1-\beta^*}$. We apply \eqref{eq:DtVt1}$_1$ and get the following inequality
\[\Dt'(\tau_*)\leq -(1-\beta^*) \Dt(\tau_*) +\Vt(\tau_*)
\leq -(1-\beta^*)\cdot\frac{M}{1-\beta^*}+M=0.\]
Hence, $\Dt(\tau_*+)\leq\Dt(\tau_*)$. This leads to a contradiction.

\textit{Case 2:} $(\Dt,\Vt)$ exits to the top, namely $\Dt(\tau_*)\in[0,\frac{M}{1-\beta^*}]$, $\Vt(\tau_*)=M$, and $\Vt(\tau_*+)>M$. We apply \eqref{eq:DtVt1}$_2$ and obtain
 \begin{align*}
  \Vt'(\tau_*)\leq&\,\beta^*\Vt(\tau_*)-\lambda C\Dt(\tau_*)^{-\alpha}\Vt(\tau_*)^{p-1}\leq\beta^*M-\lambda C\cdot\frac{(1-\beta^*)^\alpha}{M^\alpha}\cdot M^{p-1}\\
  =&\, \beta^*M\left(1-\frac{\lambda C(1-\beta^*)^\alpha}{\beta^*}M^{p-2-\alpha}\right)\leq 0,
 \end{align*}
where the definition of $M$ in \eqref{eq:M} ensures the last inequality. Hence, $\Vt(\tau_*+)\leq\Vt(\tau_*)$. This leads to a contradiction.

We have shown that the dynamics of $(\Dt,\Vt)$ is flowing inward at the boundary (see illustration in Figure \ref{fig:DTVT}). Therefore $(\Dt,\Vt)$ can not exit $A$ from either side of the boundary. Therefore, $(\Dt,\Vt)$ has to stay inside $A$ in all time. We conclude with \eqref{eq:DtVtbound1} with
\[\Dtbar=\frac{M}{1-\beta^*}\quad\text{and}\quad
\Vtbar=M,\]
with $\beta^*$ and $M$ defined in \eqref{eq:betasup} and \eqref{eq:M} respectively, depending on $\D_0,\V_0,p$ and $\alpha$.
\end{proof}

Theorem \ref{thm:case1} is a direct consequence of Proposition \ref{prop:invreg}. Indeed, we have
\[\D(t)\leq \Dtbar (t+1)^{1-\beta^*}\quad\text{and}\quad 
\V(t)\leq \Vtbar (t+1)^{-\beta^*},\]
which leads to \eqref{eq:case1bound}.
Since the result holds for any initial conditions $(\D_0,\V_0)$, the system has unconditional alignment. There is no guaranteed flocking in this scenario due to the nonlinearity. However, we obtain a bound on the growth of $\D(t)$ that is sub-linear in time.

\subsection{Scenario 2: Unconditional flocking and alignment} When $2<p<3$, the heuristic argument suggests the asymptotic flocking and alignment phenomena \eqref{eq:case1bound}. We will show these behaviors are \emph{unconditional}. 

\begin{theorem}\label{thm:case2}
Let $p\in(2,3)$ and $\phi$ satisfies \eqref{eq:alpha} with $\alpha\in[0,1)$. 
Take any bounded $(\D_0,\V_0)$. Suppose $(\D,\V)$ satisfies \eqref{eq:ODI}. Then, we have
\begin{equation}\label{eq:case2bound}
 \D(t)\leq \Dbar,\quad \V(t)\lesssim t^{-\frac{1}{p-2}}.
\end{equation}
\end{theorem}

Similarly to Scenario 1, we start with an appropriate time scaling on $(\D,\V)$. We shall only rescale $\V$ and define
\[\Vt(t)=(t+1)^{\beta_*} \V(t),\]
where we denote
\begin{equation}\label{eq:betasub}
 \beta_*=\frac{1}{p-2}.	
\end{equation}

Our goal is to bound $(\D,\Vt)$. We shall construct an invariant region
\[A = [0,\Dbar]\times[0,\Vtbar],\]
such that $(\D(t),\Vt(t))$ can not exit.

Unlike Scenario 1, since we do not scale $\D$, we can not find $\Dbar$ such that the dynamics if flowing inward at the boundary. Instead, the following bound holds as long as $(\D,\Vt)$ stays inside $A$
\[\D(t)\leq \D_0+\int_0^t\V(s)\,ds=\D_0+\int_0^t(s+1)^{-\beta_*}\Vt(s)\,ds\leq\D_0+\frac{\Vtbar}{\beta_*-1}.\]
Hence, $(\D,\Vt)$ can not exit to the right if we have
\begin{equation}\label{eq:Dbound}
 \D_0+\frac{\Vtbar}{\beta_*-1}\leq\Dbar.
\end{equation}

To argue that $(\D,\Vt)$ can not exit to the top, we compute the dynamics of $\Vt$ as in \eqref{eq:Vtprime} and get
\[\Vt'(t) \leq \frac{1}{t+1}\big(\beta_* \Vt(t)-\lambda C\D(t)^{-\alpha}\Vt(t)^{p-1}\big)\leq \frac{1}{t+1}\beta_* \Vt(t)\left(1-\frac{\lambda C}{\beta_*\Dbar^\alpha}\Vt(t)^{p-2}\right).\]
Therefore, the same argument in Case 2 of Proposition \ref{prop:invreg} implies that $(\D,\Vt)$ can not exit to the top of $A$ if we pick
\begin{equation}\label{eq:Vbound}
 \Vtbar\geq\max\left\{\V_0,\left(\frac{\beta_*\Dbar^\alpha}{\lambda C}\right)^{\beta_*}\right\}.
\end{equation}
If we can find $(\Dbar,\Vtbar)$ such that \eqref{eq:Dbound} and \eqref{eq:Vbound} hold, then $A$ is an invariant region.

Observe that the two conditions \eqref{eq:Dbound} and \eqref{eq:Vbound} imply
\begin{equation}\label{eq:DVbound}
\Vtbar\lesssim\Dbar\lesssim\Vtbar^{\frac{p-2}{\alpha}}=\Vtbar^{\frac{1}{\alpha\beta_*}}.	
\end{equation}
When $\frac{p-2}{\alpha}>1$, we can pick a large enough $\Vtbar$ such that both inequalities hold. Let us state the following proposition.

\begin{proposition}\label{prop:DVt}
	Let $2<p<3$ and $0<\alpha<p-2$. Then there exist $(\Dbar,\Vtbar)$ such that \eqref{eq:Dbound} and \eqref{eq:Vbound} hold.
\end{proposition}
\begin{proof}
 Let $\Dbar=\frac{2\Vtbar}{\beta_*-1}$. We will pick $\Vtbar$ such that \eqref{eq:Dbound} and \eqref{eq:Vbound} hold.
 
 First, if $\Vtbar\geq(\beta_*-1)\D_0$ we have
 \[\D_0+\frac{\Vtbar}{\beta_*-1}\leq \frac{2\Vtbar}{\beta_*-1}=\Dbar.\]
Next, since $p-2-\alpha>0$, we have
\begin{equation}\label{eq:Vtboundgtr}
\Vtbar\geq\left(\frac{\beta_*\Dbar^\alpha}{\lambda C}\right)^{\frac{1}{p-2}}=\left(\frac{\beta_*2^\alpha}{\lambda C(\beta_*-1)^\alpha}\right)^{\frac{1}{p-2}}\Vtbar^{\frac{\alpha}{p-2}}\quad\Leftrightarrow\quad
\Vtbar\geq \left(\frac{\beta_*2^\alpha}{\lambda C(\beta_*-1)^\alpha}\right)^{\frac{1}{p-2-\alpha}}.
\end{equation}
Hence, \eqref{eq:Dbound} and \eqref{eq:Vbound} hold if we pick
\[
\Vtbar=\max\left\{\V_0,(\beta_*-1)\D_0,\left(\frac{\beta_*2^\alpha}{\lambda C(\beta_*-1)^\alpha}\right)^{\frac{1}{p-2-\alpha}}\right\}\quad\text{and}\quad
\Dbar=\frac{2\Vtbar}{\beta_*-1}.
\]
\end{proof}

Note that Proposition \ref{prop:DVt} fails for $\alpha\in(p-2,1)$. Indeed, when $\frac{p-2}{\alpha}<1$, \eqref{eq:Vtboundgtr} becomes 
\[\Vtbar\leq \left(\frac{\beta_*2^\alpha}{\lambda C(\beta_*-1)^\alpha}\right)^{\frac{1}{p-2-\alpha}}.\]
Hence, we are not able to find $\Vtbar$ if $\V_0>\left(\frac{\beta_*2^\alpha}{\lambda C(\beta_*-1)^\alpha}\right)^{\frac{1}{p-2-\alpha}}$. 

To obtain unconditional flocking and alignment (namely show Theorem \ref{thm:case2} for any initial data), we need to upgrade our method. The idea is the following. We start with a sub-optimal scaling on $\V$ and show that $\V(t)\lesssim (t+1)^{-\beta}$ for some $\beta\in(1,\beta_*)$. This will lead to flocking: $\D(t)\leq\Dbar$. Then we can obtain the optimal decay rate $\beta_*$ applying Theorem \ref{thm:alpha0}.

\begin{proof}[Proof of Theorem \ref{thm:case2}]
Given any $\beta\in(1,\beta_*)$, we rescale $\V$ and define
\[\Vt(t)=(t+1)^\beta \V(t).\]
We will construct an invariant region
\[A = [0,\Dbar]\times[0,\Vtbar],\]
and show $(\D,\Vt)$ stays in $A$ in all time.

First, a similar argument as \eqref{eq:Dbound} implies that $(\D,\Vt)$ can not exit to the right if 
\begin{equation}\label{eq:Dbound2}
 \D_0+\frac{\Vtbar}{\beta-1}\leq\Dbar.
\end{equation}

Next, we focus on the condition that ensures that $(\D,\Vt)$ can not exit to the top of the invariant region $A$.
Compute
\begin{align*}
\Vt'(t)\leq&\, \frac{\beta}{t+1}\Vt(t)-\frac{\lambda C}{(t+1)^{\beta(p-2)}}\D(t)^{-\alpha}\Vt(t)^{p-1}\\
\leq&\,	\frac{\beta}{t+1}\Vt(t)\left(1-\frac{\lambda C}{\Dbar^\alpha}(t+1)^{(\beta_*-\beta)(p-2)}\Vt(t)^{p-2}\right).
\end{align*}

Fix any time $t_c$. Define 
\[\Vtbar_{t_c}=\left(\frac{\Dbar^\alpha}{\lambda C}\right)^{\beta_*}(t_c+1)^{-(\beta_*-\beta)}.\]
Then for any $t\geq t_c$, we have
\[1-\frac{\lambda C}{\Dbar^\alpha}(t+1)^{(\beta_*-\beta)(p-2)}\Vtbar_{t_c}^{p-2}\leq0.\]
Therefore, $\Vt(t)$ can not exit $[0,\Vtbar_{t_c}]$ after time $t_c$. 

To control $\Vt(t)$ before time $t_c$, we apply a rough bound $\V(t)\leq\V_0$, or equivalently
\[\Vt(t)\leq (t+1)^\beta\V_0\leq (t_c+1)^\beta\V_0,\quad\forall~t\in[0,t_c].\]

We pick the optimal $t_c=t_c^*$ where
\[t_c^*+1=\frac{\Dbar^{\alpha}}{\lambda C\V_0^{p-2}}\]
such that $(t_c^*+1)^\beta\V_0=\Vtbar_{t_c^*}$.
Then the argument above implies that $(\D,\Vt)$ can not exit to the top of the invariant region $A$ if we pick
\begin{equation}\label{eq:Vbound2}
 \Vtbar = \Vtbar_{t_c^*}=\frac{\V_0^{(\beta_*-\beta)(p-2)}}{(\lambda C)^\beta}\Dbar^{\alpha\beta}.
\end{equation}

\begin{remark}
Conditions \eqref{eq:Dbound2} and \eqref{eq:Vbound2} imply 
\begin{equation}\label{eq:DVbound2}
\Vtbar\lesssim\Dbar\lesssim\Vtbar^{\frac{1}{\alpha\beta}}.	
\end{equation}
This improves the bounds in \eqref{eq:DVbound}. In particular, we can choose $\beta$ such that $\frac{1}{\alpha\beta}>1$ such that a large enough $\Vtbar$ can ensure both inequalities hold.
\end{remark}

Now we find $(\Dbar,\Vtbar)$ such that \eqref{eq:Dbound2} and \eqref{eq:Vbound2} hold. Plug in \eqref{eq:Vbound2} to 
\eqref{eq:Dbound2}, we get the condition
\begin{equation}\label{eq:Dbound3}
\D_0+\frac{\V_0^{(\beta_*-\beta)(p-2)}}{(\beta-1)(\lambda C)^\beta}\Dbar^{\alpha\beta}\leq\Dbar.	
\end{equation}
Pick $\beta$ such that
\[\beta\in\left(1,\min\{\beta_*, \tfrac{1}{\alpha}\}\right).\]
Since $\alpha\beta<1$, a large enough $\Dbar$ will satisfy \eqref{eq:Dbound3}, for any given $(\D_0,\V_0)$. Indeed, we may pick 
\[\Dbar=\max\left\{2\D_0, \left(\frac{2\V_0^{1-\beta(p-2)}}{(\beta-1)(\lambda C)^\beta}\right)^{\frac{1}{1-\alpha\beta}}\right\},\]
and $\Vtbar$ from \eqref{eq:Vbound2}.

We have shown that with our choice of $(\Dbar,\Vtbar)$, the dynamics $(\D,\Vt)$ stays in the invariant region $A$ in all time. This implies the flocking phenomenon
\[\D(t)\leq\Dbar,\quad \forall~t\geq 0.\]

Finally, we can repeat the proof of Theorem \ref{thm:alpha0} with $\phim=\phi(\Dbar)$ and conclude that
$\V(t)\lesssim t^{-\beta_*}$.
\end{proof}

\subsection{The borderline scenario: Logarithmic growth}
From the heuristics, we expect $\V(t)\sim t^{-1}$ when $p=3$. This implies a possible logarithmic growth on $\D(t)$, that may then further affect the decay rate of $\V$. 

\begin{theorem}\label{thm:case3}
Let $p=3$ and $\phi$ satisfies \eqref{eq:alpha} with $\alpha\in[0,1)$. 
Take any bounded $(\D_0,\V_0)$. Suppose $(\D,\V)$ satisfies \eqref{eq:ODI}. Then, we have
\begin{equation}\label{eq:case3bound}
 \D(t)\lesssim (\log t)^{\frac{1}{1-\alpha}},\quad \V(t)\lesssim t^{-1}(\log t)^{\frac{\alpha}{1-\alpha}}.
\end{equation}
\end{theorem}

\begin{remark}
 We see from \eqref{eq:case1bound} that the power on the logarithmic correction depends on $\alpha$. In particular, when $\alpha=0$, we have $\D(t)\lesssim\log t$ and $\V(t)\lesssim t^{-1}$. This coincides with the result in Theorem \ref{thm:alpha0}. The power that we obtained is sharp.
\end{remark}

\begin{proof}
 We start with the scaling in \eqref{eq:DVscale} and follow the same procedure that leads to \eqref{eq:DtVt1}. Since $p=3$, we have $\beta^*=1$. So there is no scaling on $\D$ and
 \[\Vt(t)=(t+1)\V(t).\]
 Then for this special case, \eqref{eq:DtVt1} reads
 \begin{equation}\label{eq:DtVt3}
   \begin{cases}
    \D'(\tau) \leq \Vt(\tau),\\
    \Vt'(\tau) \leq \Vt(\tau)-\lambda C\D(\tau)^{-\alpha}\Vt(\tau)^2,
   \end{cases}
	\quad\text{with}\quad
  \begin{cases}
   \D(0)=\D_0,\\
   \Vt(0)=\V_0.
  \end{cases}
 \end{equation}
 
Now we perform another time scaling on $\tau$
\[\Dtt(\tau)=(\tau+1)^{-(\gamma+1)}\D(\tau),\quad
 \Vtt(\tau)=(\tau+1)^{-\gamma}\Vt(\tau).\]
Note that since $\tau=\log(t+1)$, the scaling above is logarithmic in $t$. The power $\gamma$ will be determined later in \eqref{eq:gamma}. Apply \eqref{eq:DtVt3} and compute
\begin{align*}
 \Dtt'(\tau)=&\,-(\gamma+1)(\tau+1)^{-(\gamma+2)}\D(\tau)+(\tau+1)^{-(\gamma+1)}\D'(\tau)\\
 \leq&\,\frac{1}{\tau+1}\big(-(\gamma+1)\Dtt(\tau)+\Vtt(\tau)\big),\\
 \Vtt'(\tau)=&\,-\gamma(\tau+1)^{-(\gamma+1)}\Vt(\tau)+(\tau+1)^{-\gamma}\Vt'(\tau)\\
 \leq&\,-\frac{\gamma}{\tau+1}\Vtt(\tau)+\Vtt(\tau)\big(1-(\tau+1)^{\gamma-(\gamma+1)\alpha}\cdot\lambda C\Dtt(\tau)^{-\alpha}\Vtt(\tau)\big).
\end{align*}
We observe that when $\gamma-(\gamma+1)\alpha<0$, the dominate contribution to the dynamics of $\Vtt$ in large time is $\Vtt'(\tau)\leq\Vtt(\tau)$, leading to an uncontrollable exponential growth. On the other hand, when $\gamma-(\gamma+1)\alpha>0$, we have $\Vtt'(\tau)\leq0$ if $\tau$ is large enough. Hence, the optimal rate $\gamma$ is such that $\gamma-(\gamma+1)\alpha=0$, namely
\begin{equation}\label{eq:gamma}
 \gamma=\frac{\alpha}{1-\alpha}.
\end{equation}
The dynamics of $(\Dtt,\Vtt)$ satisfies
 \begin{equation}\label{eq:DttVtt}
   \begin{cases}
    \Dtt'(\tau) \leq \frac{1}{\tau+1}\big(-\frac{1}{1-\alpha}\Dtt(\tau)+\Vtt(\tau)\big),\\
    \Vtt'(\tau) \leq \Vtt(\tau)\big(1-\lambda C\Dtt(\tau)^{-\alpha}\Vtt(\tau)\big),
   \end{cases}
	\quad\text{with}\quad
  \begin{cases}
   \Dtt(0)=\D_0,\\
   \Vtt(0)=\V_0.
  \end{cases}
 \end{equation}

Next, we construct an invariant region
\[A=[0,\Dtbar]\times[0,\Vtbar],\]
such that $(\Dtt,\Vtt)$ stays in $A$ in all time.
Note that \eqref{eq:DttVtt} has the same structure as \eqref{eq:DtVt1}. Therefore, we can directly apply Proposition \ref{prop:invreg} and find $(\Dtbar,\Vtbar)$. More precisely, 
\[\Vtbar=\max\left\{\V_0,\frac{1}{1-\alpha}\D_0,\left(\frac{(1-\alpha)^{\alpha}}{\lambda C}\right)^{\frac{1}{1-\alpha}}\right\}\quad\text{and}\quad
\Dtbar=(1-\alpha)\Vtbar.\]

Finally, we conclude that
\[\D(t)\leq\Dtbar(\log(t+1)+1)^{\frac{1}{1-\alpha}},\quad
\V(t)\leq\Vtbar(t+1)^{-1}(\log(t+1)+1)^{\frac{\alpha}{1-\alpha}}.\]
This finishes the proof of \eqref{eq:case3bound}.
\end{proof}

\section{Sharpness of the decay rates}\label{sec:sharp}
In this section, we show that the decay rates of the velocity alignment that we obtain are sharp. In particular, the sharp decay rates can be achieve under the setup when there are two groups that are moving away from each other. 
For simple illustration, we consider the following \emph{two-particle} initial configuration in one-dimension 
\begin{equation}\label{eq:2particle}
\rho_0=\delta_{x=\frac{x_0}{2}}+\delta_{x=-\frac{x_0}{2}},\quad
\rho_0u_0=\frac{v_0}{2}\delta_{x=\frac{x_0}{2}}-\frac{v_0}{2}\delta_{x=-\frac{x_0}{2}},	
\end{equation}
where $x_0>0$, $v_0>0$, and $\delta_{x=x_0}$ denotes the Dirac delta function at $x_0$. So we have two particles with initial distance $x_0$ and they move away from each other with relative velocity $v_0$. One can formally check that
\[\rho(t)=\delta_{x=\frac{x(t)}{2}}+\delta_{x=-\frac{x(t)}{2}},\quad
\rho u(t)=\frac{v(t)}{2}\delta_{x=\frac{x(t)}{2}}-\frac{v(t)}{2}\delta_{x=-\frac{x(t)}{2}},\]
is a weak solution of the system \eqref{eq:main} where $(x(t), v(t))$ satisfies
\[
 \begin{cases}
  x'=v,\\ v'= -\phi(x)\Phi(v)=-\phi(x)v^{p-1},
 \end{cases}
\quad\text{with}\quad
\begin{cases}
 x(0)=x_0,\\v(0)=v_0.
\end{cases}
\]
Observe that $\D(t)=x(t)$ and $\V(t)=v(t)$. Therefore, the dynamics of $(\D,\V)$ has the same structure as \eqref{eq:ODI}, with the inequalities replaced by equalities. 
\begin{equation}\label{eq:ODE}
 \begin{cases}
  \D'(t)=\V,\\ \V'(t)= -\phi(\D(t))\V(t)^{p-1},
 \end{cases}
\quad\text{with}\quad
\begin{cases}
 \D(0)=x_0,\\ \V(0)=v_0.
\end{cases}
\end{equation}
We will study \eqref{eq:ODE} and show that the optimal decay (and growth) rates are achieved under the current setup.

Our first result considers the Scenario 1: $p>3$.
\begin{theorem}\label{thm:case1o}
Let $p>3$ and $\phi$ satisfies \eqref{eq:alpha} with $\alpha\in[0,1)$. Suppose $(\D,\V)$ satisfies \eqref{eq:ODE}, with initial data $x_0>0$ and $v_0>0$. Then, we have
\begin{equation}\label{eq:case1obound}
 \D(t)\sim t^{1-\frac{1-\alpha}{p-2-\alpha}},\quad \V(t)\sim t^{-\frac{1-\alpha}{p-2-\alpha}}.
\end{equation}
\end{theorem}
\begin{proof}
 The upper bounds are already proved in Theorem \ref{thm:case1}. We focus on the lower bounds
 \begin{equation}\label{eq:case1olbound}
 \D(t)\gtrsim t^{1-\frac{1-\alpha}{p-2-\alpha}},\quad \V(t)\gtrsim t^{-\frac{1-\alpha}{p-2-\alpha}}.
\end{equation}
Apply the scaling \eqref{eq:DVscale} and derive the following dynamics similar as \eqref{eq:DtVt1}
\begin{equation}\label{eq:DtVt1o}
   \begin{cases}
    \Dt'(\tau) = (\beta^*-1) \Dt(\tau) +\Vt(\tau),\\
    \Vt'(\tau) \geq \beta^* \Vt(\tau)-\Lambda \Dt(\tau)^{-\alpha}\Vt(\tau)^{p-1},
   \end{cases}
	\quad\text{with}\quad
  \begin{cases}
   \Dt(0)=x_0,\\
   \Vt(0)=v_0.
  \end{cases}
 \end{equation}
To obtain lower bounds on $(\Dt,\Vt)$, consider the following invariant region
\[B = \left[\frac{m}{1-\beta^*},\infty\right)\times[m,\infty),\]
where
\begin{equation}\label{eq:m}
 m=\min\left\{v_0,(1-\beta^*)x_0,\left(\frac{\beta^*}{\Lambda (1-\beta^*)^\alpha}\right)^{\frac{1}{p-2-\alpha}}\right\}>0,
\end{equation}

By definition, $(x_0,v_0)\in B$. The same argument as Proposition \ref{prop:invreg} implies that the dynamics of $(\Dt,\Vt)$ can not exit $B$. See Figure \ref{fig:DTVTo} for a quick illustration.

\begin{figure}[ht]
\begin{tikzpicture}[domain=0:5]
 \draw[<->] (0,2.5) node[above]{$\Vt$} -- (0,0) node[below left]{$0$} -- (6,0) node[right]{$\Dt$};
 \draw[dashed] (0,1) node[left]{$\Vtb=m$} -- (2,1) -- (2,0) node[below,xshift=6mm]{$\Dtb=\frac{m}{1-\beta^*}$};
 \draw (2,3) -- (2,1) -- (6,1);
 \draw[blue,->] (3,1) -- (3,1.3);
 \draw[blue,->] (4,1) -- (4,1.3);
 \draw[blue,->] (5,1) -- (5,1.3);
 \draw[red,->] (2,1.5) -- (2.3,1.5);
 \draw[red,->] (2,2) -- (2.3,2);
 \draw[red,->] (2,2.5) -- (2.3,2.5);
 \draw[red!70,dashed] plot(\x,\x/2) node[right] {{\footnotesize$(\beta^*-1) \Dt+\Vt=0$}};
 \draw[blue!70,dashed] plot(\x,{(\x/2)^(1/2)}) node[right] {{\footnotesize$\beta^* \Vt-\Lambda\Dt^{-\alpha}\Vt^{p-1}=0$}};
 \draw node at (4,2) {$B$};
\end{tikzpicture}
\caption{An illustration of the invariant region $B$ defined in \eqref{eq:invreg}. To the left of the red line, $\Dt'\geq0$. Below the blue curve, $\Vt'\geq0$. Hence, trajectories can not exit $B$.} \label{fig:DTVTo}
\end{figure}

Therefore, we obtain
\[\D(t)\geq\Dtb\,(t+1)^{1-\beta^*},\quad\text{and}\quad
\V(t)\geq\Vtb\,(t+1)^{-\beta^*},\]
with
\[\Dtb=\frac{m}{1-\beta^*}>0,\quad\text{and}\quad  \Vtb=m>0.\]
This finishes the proof of \eqref{eq:case1olbound}.
\end{proof}

Next, we consider the Scenario 2: $2<p<3$.
\begin{theorem}\label{thm:case2o}
Let $p\in(2,3)$ and $\phi$ satisfies \eqref{eq:alpha} with $\alpha\in[0,1)$. Suppose $(\D,\V)$ satisfies \eqref{eq:ODE}, with initial data $x_0>0$ and $v_0>0$. Then, we have
\begin{equation}\label{eq:case2obound}
 \D(t)\sim 1,\quad \V(t)\sim t^{-\frac{1}{p-2}}.
\end{equation}
\end{theorem}
\begin{proof}
We start with a trivial bound $\D'(t)=\V(t)\geq0$. This leads to a lower bound $\D(t)\geq x_0>0$. Then from \eqref{eq:ODE}$_2$ we obtain
\[\V'(t)\geq -\phi(x_0)\V(t)^{p-1}\quad\Rightarrow\quad
\V(t)\geq \left((p-2)\phi(x_0)\,t+v_0^{-(p-2)}\right)^{-\frac{1}{p-2}}.\]
Together with Theorem \ref{thm:case2}, we conclude with \eqref{eq:case2obound}.
\end{proof}

Finally, we state the result on the borderline scenario $p=3$. The proof is similar to Theorem \ref{thm:case1o}. We omit the details.
\begin{theorem}\label{thm:case3o}
Let $p=3$ and $\phi$ satisfies \eqref{eq:alpha} with $\alpha\in[0,1)$. Suppose $(\D,\V)$ satisfies \eqref{eq:ODE}, with initial data $x_0>0$ and $v_0>0$. Then, we have
\begin{equation}\label{eq:case3obound}
  \D(t)\sim (\log t)^{\frac{1}{1-\alpha}},\quad \V(t)\sim t^{-1}(\log t)^{\frac{\alpha}{1-\alpha}}.
\end{equation}
\end{theorem}

\section{Thin tail communications: Conditional flocking and alignment}\label{sec:thin}
In this section, we move to the case when the communication protocol has a thin tail, that is, $\phi$ satisfies \eqref{eq:alpha} with $\alpha>1$.

As stated in Theorem \ref{thm:Lyap}, a major feature of the thin tail communications is that the flocking and alignment are  \emph{conditional}, namely for a class of subcritical initial data. We will show the phenomenon using the method of invariant region. 

For $p\in(2,3)$, we obtain a subcritical region $S$, defined in \eqref{eq:subDV}, that greatly enlarges the area in \eqref{eq:Lyapsub}. We further show that that the flocking and alignment are \emph{semi-unconditional} (see Definition \ref{def:semi}). On the other hand, we also construct supercritical initial data that lead to no alignment. For $p>3$, we show that this is no alignment regardless of how small $\D_0$ and $\V_0$ are.

%

\subsection{Scenario 3: Flocking and alignment for subcritical initial data}
Let us start our discussion on the case when $p\in(2,3)$. We obtain the conditional flocking and alignment result.

\begin{theorem}\label{thm:case4}
Let $p\in(2,3)$ and $\phi$ satisfies \eqref{eq:alpha} with $\alpha>1$. Suppose $(\D,\V)$ satisfies \eqref{eq:ODI}. There exists a subcritical region $S\in\R_+\times\R_+$ such that for any $(\D_0,\V_0)\in S$, we have
\begin{equation}\label{eq:case4bound}
 \D(t)\leq \Dbar,\quad \V(t)\lesssim t^{-\frac{1}{p-2}}.
\end{equation}
\end{theorem}

\begin{proof}
We follow the proof of Theorem \ref{thm:case2} until reaching the inequality \eqref{eq:Dbound3}.
Since $\alpha>1$ and $\beta>1$, there might not exist $\Dbar$ that satisfies \eqref{eq:Dbound3}. Indeed, if we view \eqref{eq:Dbound3} as 
\[f(\Dbar)=\D_0+\frac{\V_0^{1-\beta(p-2)}}{(\beta-1)(\lambda C)^\beta}\Dbar^{\alpha\beta}-\Dbar\leq0.	
\]
One can easily check that $f$ attains its minimum at
\[\Dbar=\left(\frac{(\beta-1)(\lambda C)^\beta}{\alpha\beta\V_0^{1-\beta(p-2)}}\right)^{\frac{1}{\alpha\beta-1}},\]
with
\[f_{\min}=f(\Dbar)=\D_0-(\alpha\beta-1)\left(\frac{(\beta-1)(\lambda C)^\beta}{(\alpha\beta)^{\alpha\beta}}\right)^{\frac{1}{\alpha\beta-1}}\V_0^{-\frac{1-\beta(p-2)}{\alpha\beta-1}}.\]
Therefore, if $(\D_0,V_0)$ satisfies
\begin{equation}\label{eq:case4condition}
	\D_0\V_0^{\frac{1-\beta(p-2)}{\alpha\beta-1}}\leq(\alpha\beta-1)\left(\frac{(\beta-1)(\lambda C)^\beta}{(\alpha\beta)^{\alpha\beta}}\right)^{\frac{1}{\alpha\beta-1}},
\end{equation}
then \eqref{eq:Dbound3} holds, and we have $\D(t)\leq\Dbar$. Applying Theorem \ref{thm:alpha0}, we conclude that $\V(t)\lesssim t^{-\frac{1}{p-2}}$.

Since the argument above works for any $\beta\in(1,\frac{1}{p-2})$, we can define the subcritical region $S$ as
\begin{equation}\label{eq:subDV}
 S = \left\{(\D_0,\V_0)\in\R_+\times\R_+ : \exists~\beta\in(1,\tfrac{1}{p-2}) \text{ such that \eqref{eq:case4condition} holds}\right\}.
\end{equation}
\end{proof}

\begin{figure}[ht]
 \includegraphics{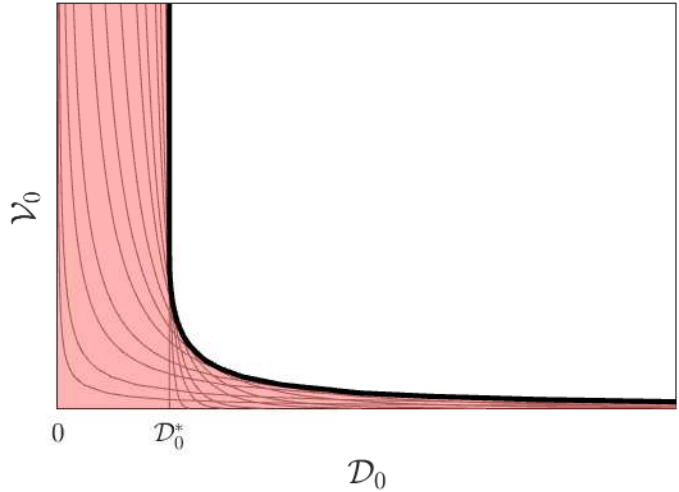}	
 \caption{An illustration of the subcritical region $S$}\label{fig:sub}
\end{figure}

Let us comment on the subcritical region $S$. Figure \ref{fig:sub} provides an illustration of $S$ in \eqref{eq:subDV}. It is a union of the regions in \eqref{eq:case4condition} by varying $\beta\in(1,\frac{1}{p-2})$. A surprising observation is: $(\D_0,\V_0)\in S$ as long as $\D_0<\D_0^*$, regardless of how big $\V_0$ is. We state the following proposition.
\begin{proposition}\label{prop:sub}
 The region $S$ defined in \eqref{eq:subDV} satisfies
 \[[0,\D_0^*)\times\R_+\subset S,\quad\text{where}\quad
 \D_0^*:=(\alpha\beta_*-1)\left(\frac{(\beta_*-1)(\lambda C)^{\beta_*}}{(\alpha\beta_*)^{\alpha\beta_*}}\right)^{\frac{1}{\alpha\beta_*-1}}.\]
\end{proposition}
The proposition can be proved by taking $\beta\to\beta_*$ in \eqref{eq:case4condition}, observing that the power of $\V_0$ becomes $\lim_{\beta\to\beta_*}\frac{1-\beta(p-2)}{\alpha\beta-1}=0$, and the right hand side of \eqref{eq:case4condition} is continuous in $\beta$. We omit the detailed proof.

Proposition \ref{prop:sub} implies \emph{semi-unconditional} flocking and alignment: if $\D_0\in(0,\D_0^*)$, for any $\V_0>0$, we apply Theorem \ref{thm:case4} and obtain flocking and alignment \eqref{eq:case4bound}. 

\begin{remark}
 The result can be extended to the linear case $p=2$. Indeed, we have $\beta_*=\infty$. When taking $\beta_*\to\infty$, we get $\D_0^*\to(\lambda C)^{1/\alpha}$ in Proposition \ref{prop:sub}. Hence, if $\D_0\in(0,(\lambda C)^{1/\alpha})$, for any $\V_0>0$, we obtain flocking and fast alignment. Therefore, we conclude that the asymptotic behaviors are semi-unconditional. 	
\end{remark}

\subsection{Scenario 3: No alignment for supercritical initial data}
In this part, we construct supercritical initial data that lead to no alignment, that is $\lim_{t\to\infty}\V(t)\neq0$. It indicates that the flocking and alignment are indeed conditional.

We use the two-particle initial configuration \eqref{eq:2particle}.

\begin{theorem}\label{thm:case4o}
Let $p\in(2,3)$ and $\phi$ satisfies \eqref{eq:alpha} with $\alpha>1$. Suppose $(\D,\V)$ satisfies \eqref{eq:ODE}. There exists a supercritical region $T\in\R_*\times\R_*$, for any $(x_0,v_0)\in T$, there exist $\Dtb>0$ and $\Vb>0$ such that 
\begin{equation}\label{eq:case1obound}
 \D(t)\geq\Dtb(t+1),\quad\text{and}\quad \V(t)\geq\Vb>0.
\end{equation}
\end{theorem}

\begin{proof}
 We start with applying the following scaling on $\D$ to \eqref{eq:ODE}
 \[\Dt(t)=(t+1)^{-1}\D(t),\]
 and compute
\begin{align}
\Dt'(t)=&-(t+1)^{-2}\D(t)+(t+1)^{-1}\D'(t)=(t+1)^{-1} \big(-\Dt(t) +\V(t)\big),\label{eq:case4D}\\
 \V'(t)\geq & -\Lambda\D(t)^{-\alpha}\V(t)^{p-1}=-(t+1)^{-\alpha}\Lambda\Dt(t)^{-\alpha}\V(t)^{p-1}.\label{eq:case4V}
\end{align}

We will construct an invariant region
\[B = [\Dtb,\infty)\times[\Vb,\infty)\]
and argue that the dynamics of $(\Dt,\V)$ stays in $B$ in all time.

First, we check that $(\Dt,\V)$ can not exit from below. From \eqref{eq:case4V} we get
\[\V'(t)\geq -(t+1)^{-\alpha}\Lambda\Dtb^{-\alpha}\V(t)^{p-1},\]
as long as $\Dt(t)\geq\Dtb$. This can be further simplified using separation of variables
\[\V(t)\geq \left(\frac{1}{v_0^{p-2}}+\frac{(p-2)\Lambda\Dtb^{-\alpha}}{\alpha-1}\big(1-(t+1)^{-(\alpha-1)}\big)\right)^{-\frac{1}{p-2}}\geq\left(\frac{1}{v_0^{p-2}}+\frac{(p-2)\Lambda\Dtb^{-\alpha}}{\alpha-1}\right)^{-\frac{1}{p-2}}.\]
Therefore, if we pick 
\begin{equation}\label{eq:Dboundo}
 \Vb=\left(\frac{1}{v_0^{p-2}}+\frac{(p-2)\Lambda\Dtb^{-\alpha}}{\alpha-1}\right)^{-\frac{1}{p-2}}>0,
\end{equation}
then $\V(t)$ can not drop below $\Vb$.

Next, we check that $(\Dt,\V)$ can not exit to the left if
\begin{equation}\label{eq:Vboundo}
\Dtb\leq\min\{x_0,\Vb\}.	
\end{equation}
Indeed, if there exists a time $t_*$ such that $\Dt(t_*)=\Dtb$, $\Dt(t_*+)<\Dtb$ and $\V(t_*)\geq\Vb$. Then \eqref{eq:case4D} implies $\Dt'(t_*)\leq0$, which leads to a contradiction.

Conditions \eqref{eq:Dboundo} and \eqref{eq:Vboundo} guarantee that $(\Dt,\V)$ stays in the invariant region $B$. This directly implies \eqref{eq:case1obound}.

We are left to find $(\Dtb,\Vb)$ that satisfies \eqref{eq:Dboundo} and \eqref{eq:Vboundo}. Rewrite the conditions as
\begin{equation}\label{eq:ffunc}
f(\Dtb):=v_0^{-(p-2)}\Dtb^\alpha-\Dtb^{\alpha-(p-2)}+\frac{(p-2)\Lambda}{\alpha-1}\leq0\quad\text{and}\quad \Dtb\leq x_0.	
\end{equation}
The inequality $f(\Dtb)\leq0$ has a solution if and only if
\[f_{\min}=-\frac{p-2}{\alpha}\left(\frac{\alpha-(p-2)}{\alpha}\right)^{\frac{\alpha-(p-2)}{p-2}}v_0^{\alpha-(p-2)}+\frac{(p-2)\Lambda}{\alpha-1}\leq0,\]
or equivalently
\begin{equation}\label{eq:supv}
v_0\geq \left(\frac{\alpha\Lambda}{\alpha-1}\right)^{\frac{\beta_*}{\alpha\beta_*-1}}\left(\frac{\alpha\beta_*}{\alpha\beta_*-1}\right)^{\beta_*}.
\end{equation}
where the minimum of $f$ is achieved at
\begin{equation}\label{eq:Dboundoo}
\Dtb=\left(\frac{\alpha-(p-2)}{p-2}\right)^{\frac{1}{p-2}}v_0=(\alpha\beta_*-1)^{\beta_*}v_0.	
\end{equation}
To make sure $f(\Dtb)\leq0$ has a solution in $[0,x_0]$, we need $x_0$ to be large enough. A sufficient condition is
\begin{equation}\label{eq:supx}
	x_0\geq(\alpha\beta_*-1)^{\beta_*}v_0.
\end{equation}

Define the supercritical region
\begin{equation}\label{eq:supDV}
 T=\left\{(x_0,v_0)\in\R_+\times\R_+ :\text{ \eqref{eq:supv} and \eqref{eq:supx} holds}\right\}.	
\end{equation}
We conclude that if the initial data $(x_0,v_0)\in T$, then we can find $(\Dtb,\Vb)$ in \eqref{eq:Dboundoo} and \eqref{eq:Dboundo} respectively such that \eqref{eq:case1obound} holds.
\end{proof}

\subsection{Scenario 4: No alignment for generic data}
Now we turn our attention to the scenario when $p>3$. Observe that the Theorem \ref{thm:case4o} can be directly extended to any $p-2<\alpha$. On the other hand, if $p-2>\alpha$, we are able to obtain a stronger result.
\begin{corollary}\label{cor:case5o}
 Let $p>\alpha+2$ and $\phi$ satisfies \eqref{eq:alpha} with $\alpha>1$. Suppose $(\D,\V)$ satisfies \eqref{eq:ODE}. Then, for any initial data $x_0>0$ and $v_0>0$, there exist $\Dtb>0$ and $\Vb>0$ such that \eqref{eq:case1obound} holds.
\end{corollary}

\begin{proof}
 We follow the same proof in Theorem \ref{thm:case4o} and reach \eqref{eq:ffunc}. Note that $\alpha-(p-2)<0$. Therefore, we observe that $f$ is continuous and increasing in $(0,\infty)$, with
 \[\lim_{\Dtb\to0+}f(\Dtb)=-\infty,\quad\text{and}\quad\lim_{\Dtb\to\infty}f(\Dtb)=\infty.\]
Hence, $f$ has a unique root $\Dtb_*>0$, depending on $v_0$ and the parameters $\alpha, p, \Lambda$, such that $f(\Dtb)\leq0$ for any $\Dtb\in(0,\Dtb_*]$.
Then \eqref{eq:ffunc} is satisfied if we choose
\[\Dtb=\min\{\Dtb_*,x_0\}>0.\]
We further choose $\Vb$ according to \eqref{eq:Dboundo}. The same argument in Theorem \ref{thm:case4o} leads to \eqref{eq:case1obound}.
\end{proof}

Corollary \ref{cor:case5o} indicates that the supercritical region $T=\R_+\times\R_+$. So there is \emph{no alignment} for generic data, no matter how small the initial data are.

A natural question is on the case where $3<p\leq \alpha+2$: whether there is no alignment for all data, or there exists subcritical region that leads to alignment. The following theorem gives a comprehensive answer.
\begin{theorem}\label{thm:case5o}
 Let $p>3$ and $\phi$ satisfies \eqref{eq:alpha} with $\alpha>1$. Suppose $(\D,\V)$ satisfies \eqref{eq:ODE}. Then, for any initial data $x_0>0$ and $v_0>0$, there exist $\Dtb>0$ and $\Vb>0$ such that \eqref{eq:case1obound} holds.
\end{theorem}

The proof of Theorem \ref{thm:case5o} requires an upgrade to Theorem \ref{thm:case4o}. We use a similar idea as in the proof of Theorem \ref{thm:case2}: start with a sub-optimal scaling on $\D$ and show $\D(t)\geq (t+1)^{\gamma}$ for some $\gamma\in(\frac{1}{\alpha},1)$. This will lead to no alignment: $\V(t)\geq\Vb$. Then we can obtain the optimal growth on $\D$.

\begin{proof}[Proof of Theorem \ref{thm:case5o}]
 Given any $\gamma\in(\frac{1}{\alpha},1)$, we rescale $\D$ and define
 \[\Dt(t)=(t+1)^{-\gamma}\D(t).\]
 We will construct an invariant region
 \[B=[\Dtb,\infty)\times[\Vb,\infty),\]
 and show $(\Dt,\V)$ stays in $B$ in all time.
 
 To check $(\Dt,\V)$ can not exit from below, we compute
 \[\V'(t)\geq-(t+1)^{-\gamma\alpha}\Lambda\Dtb^{-\alpha}\V(t)^{p-1},\]
 and it implies
 \begin{equation}\label{eq:Vcondo}
\V(t)\geq\left(\frac{1}{v_0^{p-2}}+\frac{(p-2)\Lambda\Dtb^{-\alpha}}{\gamma\alpha-1}\right)^{-\frac{1}{p-2}}=:\Vb. 	
 \end{equation}
 Clearly, $\V$ can not drop below $\Vb$ (defined as the right hand side of the above inequality).
 
 Next, we focus on the condition that ensures $(\Dt,\V)$ can not exit to the left of $B$. Compute
 \[\Dt'(t)=-\frac{\gamma}{t+1}\Dt(t)+\frac{1}{(t+1)^\gamma}\V(t)\geq-\frac{\gamma}{t+1}\big(\Dt(t)-(t+1)^{1-\gamma}\Vb\big).\]
 Fix any $t_c>0$. Define
 \[\Dtb_{t_c}=(t_c+1)^{1-\gamma}\Vb.\]
 Then for any $t\geq t_c$, we have
 \[-\frac{\gamma}{t+1}\big(\Dtb_{t_c}-(t+1)^{1-\gamma}\Vb\big)\geq0.\]
 Therefore, $\Dt(t)$ can not exit $[\Dtb_{t_c},\infty)$ after time $t_c$.
 
 To control $\Dt(t)$ before time $t_c$, we apply the rough bound $\D(t)\geq x_0$. Then
 \[\Dt(t)\geq (t+1)^{-\gamma}x_0\geq (t_c+1)^{-\gamma}x_0.\]
 
 We pick the optimal $t_c=t_c^*$ where
 \[t_c^*+1=\frac{x_0}{\Vb},\]
 such that $(t_c+1)^{-\gamma}x_0=\Dtb_{t_c}$. Then the argument above implies that $(\Dt,\V)$ can not exit to the left of the invariant region $B$ if we pick
 \begin{equation}\label{eq:Dcondo}
  \Dtb\leq x_0^{1-\gamma}\Vb^{-\gamma}.	
 \end{equation}
 
 We are left to find $(\Dtb,\Vb)$ that satisfies \eqref{eq:Vcondo} and \eqref{eq:Dcondo}. Rewrite the conditions as
\begin{equation}\label{eq:ffunc2}
f(\Dtb):=v_0^{-(p-2)}\Dtb^\alpha-x_0^{-\frac{(1-\gamma)(p-2)}{\gamma}}\Dtb^{\alpha-\frac{p-2}{\gamma}}+\frac{(p-2)\Lambda}{\gamma\alpha-1}\leq0.
\end{equation}
Note that \eqref{eq:ffunc2} reduces to \eqref{eq:ffunc} if $\gamma=1$. The major gain here is that we can choose $\gamma\in(\frac{1}{\alpha},\frac{p-2}{\alpha})$ such that the power of the second term
\[\alpha-\frac{p-2}{\gamma}<0.\]
Then we use the same argument in Proposition \ref{prop:sub} and pick $\Dtb$ to be the root of $f(\Dtb)=0$, which depends on $x_0, v_0$ and parameters $p,\alpha,\Lambda,\gamma$. One can check that $\Dtb>0$, for any choice of $x_0>0$ and $v_0>0$.
Choosing $\Vb$ according to \eqref{eq:Vcondo}, we conclude that $\V(t)\geq\Vb$.

Once we obtain the lower bound on $\V$, we immediately have
\[\D(t)=x_0+\int_0^t\V(s)\,ds\geq x_0+\Vb t.\]
This concludes the proof of \eqref{eq:case1obound}, with $\Dtb=\min\{x_0,\Vb\}$.
\end{proof}

\bibliographystyle{plain}
\bibliography{bib_NonlinearEA}

\begin{thebibliography}{10}

\bibitem{an2020global}
Jing An and Lenya Ryzhik.
\newblock Global well-posedness for the {E}uler alignment system with mildly
  singular interactions.
\newblock {\em Nonlinearity}, 33(9):4670, 2020.

\bibitem{arnaiz2021singularity}
Victor Arnaiz and {\'A}ngel Castro.
\newblock Singularity formation for the fractional {E}uler-alignment system in
  {1D}.
\newblock {\em Transactions of the American Mathematical Society},
  374(1):487--514, 2021.

\bibitem{bai2022global}
Xiang Bai, Qianyun Miao, Changhui Tan, and Liutang Xue.
\newblock Global well-posedness and asymptotic behavior in critical spaces for
  the compressible {E}uler system with velocity alignment.
\newblock {\em arXiv preprint arXiv:2207.02429}, 2022.

\bibitem{carrillo2016critical}
Jos{\'e}~A Carrillo, Young-Pil Choi, Eitan Tadmor, and Changhui Tan.
\newblock Critical thresholds in {1D} {E}uler equations with non-local forces.
\newblock {\em Mathematical Models and Methods in Applied Sciences},
  26(01):185--206, 2016.

\bibitem{chen2021global}
Li~Chen, Changhui Tan, and Lining Tong.
\newblock On the global classical solution to compressible {E}uler system with
  singular velocity alignment.
\newblock {\em Methods and Applications of Analysis}, 28(2):155--174, 2021.

\bibitem{choi2019global}
Young-Pil Choi.
\newblock The global {C}auchy problem for compressible {E}uler equations with a
  nonlocal dissipation.
\newblock {\em Mathematical Models and Methods in Applied Sciences},
  29(01):185--207, 2019.

\bibitem{constantin2020entropy}
Peter Constantin, Theodore~D Drivas, and Roman Shvydkoy.
\newblock Entropy hierarchies for equations of compressible fluids and
  self-organized dynamics.
\newblock {\em SIAM Journal on Mathematical Analysis}, 52(3):3073--3092, 2020.

\bibitem{cucker2007emergent}
Felipe Cucker and Steve Smale.
\newblock Emergent behavior in flocks.
\newblock {\em IEEE Transactions on automatic control}, 52(5):852--862, 2007.

\bibitem{danchin2019regular}
Rapha{\"e}l Danchin, Piotr~B Mucha, Jan Peszek, and Bartosz Wr{\'o}blewski.
\newblock Regular solutions to the fractional {E}uler alignment system in the
  {B}esov spaces framework.
\newblock {\em Mathematical Models and Methods in Applied Sciences},
  29(01):89--119, 2019.

\bibitem{do2018global}
Tam Do, Alexander Kiselev, Lenya Ryzhik, and Changhui Tan.
\newblock Global regularity for the fractional {E}uler alignment system.
\newblock {\em Archive for Rational Mechanics and Analysis}, 228(1):1--37,
  2018.

\bibitem{figalli2018rigorous}
Alessio Figalli and Moon-Jin Kang.
\newblock A rigorous derivation from the kinetic {C}ucker--{S}male model to the
  pressureless {E}uler system with nonlocal alignment.
\newblock {\em Analysis \& PDE}, 12(3):843--866, 2018.

\bibitem{ha2010emergent}
Seung-Yeal Ha, Taeyoung Ha, and Jong-Ho Kim.
\newblock Emergent behavior of a {C}ucker-{S}male type particle model with
  nonlinear velocity couplings.
\newblock {\em IEEE Transactions on Automatic Control}, 55(7):1679--1683, 2010.

\bibitem{ha2009simple}
Seung-Yeal Ha and Jian-Guo Liu.
\newblock A simple proof of the cucker-smale flocking dynamics and mean-field
  limit.
\newblock {\em Communications in Mathematical Sciences}, 7(2):297--325, 2009.

\bibitem{ha2008particle}
Seung-Yeal Ha and Eitan Tadmor.
\newblock From particle to kinetic and hydrodynamic descriptions of flocking.
\newblock {\em Kinetic and Related Models}, 1(3):415--435, 2008.

\bibitem{kim2020complete}
Jong-Ho Kim and Jea-Hyun Park.
\newblock Complete characterization of flocking versus nonflocking of
  {C}ucker--{S}male model with nonlinear velocity couplings.
\newblock {\em Chaos, Solitons \& Fractals}, 134:109714, 2020.

\bibitem{kiselev2018global}
Alexander Kiselev and Changhui Tan.
\newblock Global regularity for {1D} {E}ulerian dynamics with singular
  interaction forces.
\newblock {\em SIAM Journal on Mathematical Analysis}, 50(6):6208--6229, 2018.

\bibitem{lear2022geometric}
Daniel Lear, Trevor~M Leslie, Roman Shvydkoy, and Eitan Tadmor.
\newblock Geometric structure of mass concentration sets for pressureless
  {E}uler alignment systems.
\newblock {\em Advances in Mathematics}, 401:108290, 2022.

\bibitem{lear2022existence}
Daniel Lear and Roman Shvydkoy.
\newblock Existence and stability of unidirectional flocks in hydrodynamic
  {E}uler alignment systems.
\newblock {\em Analysis \& PDE}, 15(1):175--196, 2022.

\bibitem{leslie2020lagrangian}
Trevor~M Leslie.
\newblock On the {L}agrangian trajectories for the one-dimensional {E}uler
  alignment model without vacuum velocity.
\newblock {\em Comptes Rendus. Math{\'e}matique}, 358(4):421--433, 2020.

\bibitem{leslie2019structure}
Trevor~M Leslie and Roman Shvydkoy.
\newblock On the structure of limiting flocks in hydrodynamic {E}uler
  {A}lignment models.
\newblock {\em Mathematical Models and Methods in Applied Sciences},
  29(13):2419--2431, 2019.

\bibitem{leslie2021sticky}
Trevor~M Leslie and Changhui Tan.
\newblock Sticky particle {C}ucker-{S}male dynamics and the entropic selection
  principle for the {1D} {E}uler-alignment system.
\newblock {\em arXiv preprint arXiv:2108.07715}, 2021.

\bibitem{lu2022hydrodynamic}
Jingcheng Lu and Eitan Tadmor.
\newblock Hydrodynamic alignment with pressure {II}. multispecies.
\newblock {\em arXiv preprint arXiv:2208.12411}, 2022.

\bibitem{miao2021global}
Qianyun Miao, Changhui Tan, and Liutang Xue.
\newblock Global regularity for a {1D} {E}uler-alignment system with
  misalignment.
\newblock {\em Mathematical Models and Methods in Applied Sciences},
  31(03):473--524, 2021.

\bibitem{motsch2011new}
Sebastien Motsch and Eitan Tadmor.
\newblock A new model for self-organized dynamics and its flocking behavior.
\newblock {\em Journal of Statistical Physics}, 144(5):923--947, 2011.

\bibitem{shvydkoy2019global}
Roman Shvydkoy.
\newblock Global existence and stability of nearly aligned flocks.
\newblock {\em Journal of Dynamics and Differential Equations},
  31(4):2165--2175, 2019.

\bibitem{shvydkoy2021dynamics}
Roman Shvydkoy.
\newblock {\em Dynamics and analysis of alignment models of collective
  behavior}.
\newblock Springer, 2021.

\bibitem{shvydkoy2017eulerian}
Roman Shvydkoy and Eitan Tadmor.
\newblock {E}ulerian dynamics with a commutator forcing.
\newblock {\em Transactions of Mathematics and its Applications}, 1(1):tnx001,
  2017.

\bibitem{shvydkoy2017eulerian2}
Roman Shvydkoy and Eitan Tadmor.
\newblock {E}ulerian dynamics with a commutator forcing {II}: Flocking.
\newblock {\em Discrete \& Continuous Dynamical Systems}, 37(11):5503--5520,
  2017.

\bibitem{tadmor2022swarming}
Eitan Tadmor.
\newblock Swarming: hydrodynamic alignment with pressure.
\newblock {\em arXiv preprint arXiv:2208.11786}, 2022.

\bibitem{tadmor2014critical}
Eitan Tadmor and Changhui Tan.
\newblock Critical thresholds in flocking hydrodynamics with non-local
  alignment.
\newblock {\em Philosophical Transactions of the Royal Society of London A:
  Mathematical, Physical and Engineering Sciences}, 372(2028):20130401, 2014.

\bibitem{tan2019singularity}
Changhui Tan.
\newblock Singularity formation for a fluid mechanics model with nonlocal
  velocity.
\newblock {\em Communications in Mathematical Sciences}, 17(7):1779--1794,
  2019.

\bibitem{tan2020euler}
Changhui Tan.
\newblock On the {E}uler-alignment system with weakly singular communication
  weights.
\newblock {\em Nonlinearity}, 33(4):1907, 2020.

\bibitem{tan2021eulerian}
Changhui Tan.
\newblock {E}ulerian dynamics in multidimensions with radial symmetry.
\newblock {\em SIAM Journal on Mathematical Analysis}, 53(3):3040--3071, 2021.

\bibitem{tong2020global}
Lining Tong, Li~Chen, Simone G{\"o}ttlich, and Shu Wang.
\newblock The global classical solution to compressible {E}uler system with
  velocity alignment.
\newblock {\em AIMS Mathematics}, 5(6):6673--6692, 2020.

\bibitem{vazquez2016dirichlet}
Juan~Luis V{\'a}zquez.
\newblock The {D}irichlet problem for the fractional $p$-{L}aplacian evolution
  equation.
\newblock {\em Journal of Differential Equations}, 260(7):6038--6056, 2016.

\bibitem{vazquez2020evolution}
Juan~Luis V{\'a}zquez.
\newblock The evolution fractional $p$-{L}aplacian equation in {$\R^N$}.
  fundamental solution and asymptotic behaviour.
\newblock {\em Nonlinear Analysis}, 199:112034, 2020.

\bibitem{vazquez2021fractional}
Juan~Luis V{\'a}zquez.
\newblock The fractional $p$-{L}aplacian evolution equation in {$\R^N$} in the
  sublinear case.
\newblock {\em Calculus of Variations and Partial Differential Equations},
  60(4):1--59, 2021.

\bibitem{vazquez2022growing}
Juan~Luis V{\'a}zquez.
\newblock Growing solutions of the fractional $p$-{L}aplacian equation in the
  fast diffusion range.
\newblock {\em Nonlinear Analysis}, 214:112575, 2022.

\end{thebibliography}

\end{document}